\documentclass[reqno,12pt,a4paper]{amsart}
\usepackage{pgfplots}
\usepackage[english]{babel}
\usepackage{amsmath,amsthm,amssymb,amsfonts, mathdots}
\usepackage{scrtime, cancel}
\usepackage{enumitem, color}

\usepackage{mathrsfs}
\usepackage{hyperref}

\usepackage{mathtools}
\usepackage{pdfsync,verbatim}
\usepackage{color}
\mathtoolsset{showonlyrefs}

\numberwithin{equation}{section}   

\allowdisplaybreaks
\newtheorem{theorem}{Theorem}[section]
\newtheorem{lemma}[theorem]{Lemma}
\newtheorem{proposition}[theorem]{Proposition}

\newtheorem{claim}[theorem]{Claim}

\theoremstyle{definition}
\newtheorem{remark}[theorem]{Remark}

\usepackage{graphicx}

\makeatletter
\newcommand*\bigcdot{\mathpalette\bigcdot@{.5}}
\newcommand*\bigcdot@[2]{\mathbin{\vcenter{\hbox{\scalebox{#2}{$\m@th#1\bullet$}}}}}
\makeatother

\newcommand{\R}{\mathbb{R}}
\newcommand{\C}{\mathbb{C}}
\newcommand{\N}{\mathbb{N}}
\newcommand{\Z}{\mathbb{Z}}
\newcommand{\ellipses}{E}
\def\misgausskd{d \gamma_\infty}
\def\misgaussk{\gamma_\infty}
\DeclareMathOperator{\tr}{tr}             
\newcount\dotcnt\newdimen\deltay
\def\Ddot#1#2(#3,#4,#5,#6){\deltay=#6\setbox1=\hbox to0pt{\smash{\dotcnt=1
\kern#3\loop\raise\dotcnt\deltay\hbox to0pt{\hss#2}\kern#5\ifnum\dotcnt<#1
\advance\dotcnt 1\repeat}\hss}\setbox2=\vtop{\box1}\ht2=#4\box2}

\def\Blue{\color{blue}}

\makeatletter
\def\author@andify{%
  \nxandlist {\unskip ,\penalty-1 \space\ignorespaces}%
    {\unskip {} \@@and~}%
    {\unskip \penalty-2 \space \@@and~}%
}
\makeatother

\title[Spectral multipliers]
{Spectral multipliers
in a general Gaussian setting}
\author{Valentina Casarino}
\address{DTG, Universit\`a degli Studi di Padova\\ Stradella san Nicola 3 \\I-36100 Vicenza \\ Italy}
\email{valentina.casarino@unipd.it}
\author{Paolo Ciatti}
\address{Dipartimento di Matematica "Tullio Levi Civita", Universit\`a degli Studi di Padova\\Via Trieste, 63, 35131 Padova,  \\ Italy}
\email{paolo.ciatti@unipd.it}
\author{Peter Sj\"ogren}
\address{Mathematical Sciences,  University of Gothenburg and  Mathematical Sciences,
Chalmers University of Technology  \\ SE - 412 96 G\"oteborg, Sweden}
\email{peters@chalmers.se}

\date{\today}

\begin{document}

\begin{abstract}
We investigate a class of spectral multipliers
for
an  Ornstein--Uhlenbeck operator $\mathcal L$ in $\R^n$,
with drift given by a real   matrix $B$ whose eigenvalues have negative real parts.
We
prove  that if $m$ is a function of Laplace transform type defined
in the right half-plane,
then $m(\mathcal L)$
 is of weak type $(1, 1)$ with respect to the invariant measure in $\R^n$.
The proof involves
 many estimates of the relevant integral kernels and also
 a bound for
 the number of zeros of the time derivative of  the Mehler kernel,
  as well as  an enhanced version of the Ornstein--Uhlenbeck maximal operator theorem.
\end{abstract}


\thanks{The first and second authors are members of the Gruppo Nazionale per l'Analisi Matematica, la Probabilit\`a e le loro Applicazioni (GNAMPA)
of the Istituto Nazionale di Alta Matematica (INdAM) and were partially supported by GNAMPA
(Project 2020 ``Alla frontiera tra l'analisi complessa in pi\`u variabili e l'analisi armonica").
Since this work was mainly carried out during a visit  to Padova by the third author, he would like to thank the
University of Padova for its generosity.}
\maketitle

\section{Introduction}

Given a measure space $(X,\mu)$ and a self-adjoint operator $L$ on $L^2(X,\mu)$, an important issue in harmonic analysis  concerns  the boundedness
of the operator $m(L)$, where $m:\R\to\C$ is a Borel function.  If $E$ denotes a  spectral resolution of $L$ on $\R$, one can define $m(L)$ for many  functions $m$ as
\[
m(L)=
\int_{\R} m(\nu)\, dE(\nu).
\]
Great efforts have been devoted to finding minimal assumptions  on the multiplier
$m$ that will ensure the boundedness
of $m(L)$ on the Lebesgue spaces $L^p(X, \mu)$, both in a strong and in a weak sense,
when $p\neq 2$.

A few years ago, the authors started a program concerning harmonic analysis in   the Ornstein--Uhlenbeck setting.
In this framework,
$(X,\mu)$ is the Euclidean space $\R^n$ equipped with a  Gaussian measure
$d\gamma_\infty$,
known as the invariant measure and defined in Section \ref{preliminaries}.
Further,
 $L$ is replaced by the Ornstein--Uhlenbeck operator $\mathcal L$, defined as
\begin{equation}\label{def-L-OU}
\mathcal L f=- \frac12 \,
\tr
\big(
Q\nabla^2
f\big)-
\langle Bx, \nabla f \rangle
\,,\qquad
{\text{ $f\in
\mathcal S (\R^n)$,}}
\end{equation}
where
$\nabla$ and $\nabla^2$ denote  the gradient  and the Hessian, respectively.
In this formula,
$Q$ and $B$
 are real $n\times n$  matrices;
 $Q$ is symmetric and positive definite, and   the
 eigenvalues of $B$ all have   
 negative real parts.
The space  $L^p(\R^n, d\gamma_\infty)$ will be written simply
  $L^p(\gamma_\infty)$.

Since in general
$\mathcal L$ has no self-adjoint or normal extension to $L^2(\gamma_\infty)$,
one cannot invoke
 spectral theory
to define $m(\mathcal L)$.
Notice that
self-adjointness  and normality may fail  also for  the Ornstein--Uhlenbeck semigroup
$\left(
\mathcal H_t
\right)_{t> 0}$,
generated by $\mathcal L$,
which was first introduced in  \cite{OU}.
The focus in this paper is
 on
 multipliers of Laplace transform type.
 This class of multipliers was introduced some fifty years ago
by E. M. Stein in \cite{Stein-LP},  in the context of    Littlewood--Paley theory
for a sublaplacian on a connected Lie group $G$.

A function  $m$ of a real variable  $\lambda>0$
 is said to be of Laplace  transform type
if
\begin{align}\label{defmlambda}
m(\lambda)
&=\lambda \int_0^{+\infty}
\varphi (t)
e^{-t \lambda } \,dt
= -\int_0^{+\infty}
\varphi (t)\,
\frac{d}{dt}\, e^{-t\lambda }
\,dt, \qquad \lambda >0,
\end{align}
for some $\varphi\in L^\infty (0, +\infty)$.
Observe that such a function $m$ can be extended to an analytic function in the half-plane
$\Re z > 0$.
Thus we pay the price of a rather strong condition on $m$, to  prove, in return,
a multiplier theorem for an operator   $\mathcal L$ which is
not necessarily normal.
 Observe that one obtains as $m(\mathcal L)$ the imaginary powers
$\mathcal L^{i\gamma}$ of   $\mathcal L$, with $ \gamma \in \R\setminus\{0\} $, by choosing $\varphi(t) = \mathrm{const.}\, t^{-i\gamma}$.
Other significant examples of functions of Laplace transform type may be found in \cite{Wr3}.

The exact definition of  $m(\mathcal L)$ for functions $m$ of this type will be given in
Section~\ref{def}. Here we present only a heuristic deduction of the kernel of $m(\mathcal L)$. If we simply replace $\lambda$
by $\mathcal L$ in the last expression in \eqref{defmlambda}, we will get
\begin{align}\label{heur}
m(\mathcal L)
=- \int_0^{+\infty}
\varphi (t)\,
\frac{d}{dt}\, e^{-t\mathcal L }
\,dt.
\end{align}

 Here $e^{-t\mathcal L } = \mathcal H_t$ is the Ornstein-Uhlenbeck semigroup,
whose kernel is the Mehler kernel $K_t(x,u)$ described in Section \ref{preliminaries}. We point out that the term kernel  in this paper  refers to integration with respect to
 $ d\gamma_\infty$, except in one case as explained in Section \ref{local region}.
Thus for each $f\in\mathcal S(\R^n)$ and all $t>0$
\begin{align}\label{def-int-ker}
 \mathcal H_t
f(x) &=
 \int
K_t
(x,u)\,
f(u)\,
 d\gamma_\infty(u)
  \,.
\end{align}
 This makes it plausible that the off-diagonal kernel of
 $m(\mathcal L)$ is
 \begin{equation}\label{Mehler-kernel-derivative-mult}
\mathcal M_{\varphi} (x,u)
=
 -\int_0^{+\infty}
\varphi (t)\,
\partial_t K_t (x,u)
\,dt.
\end{equation}
We will verify this formula later, though after splitting the integral and under some restrictions.
It will lead to an expression for  the kernel   in terms of
$Q$ and $B$.

From now on, we assume that $m$  is  of Laplace  transform type.

In the
 standard
 case $Q=I$ and $B=-I$, the operator  $\mathcal L$ is self-adjoint, and
the $L^p(d\gamma_\infty)$ boundedness of $m(\mathcal L)$ follows for all $1<p<\infty$
from a general result due to Stein \cite[Ch. 4]{Stein-LP}.
Moreover,
 J. Garc\'ia-Cuerva, G. Mauceri,  J. L.Torrea and the third author  proved in this
  case the weak type $(1,1)$ of $m(\mathcal L)$  with respect to  $d\gamma_\infty$\hspace{0.03cm}; see  \cite[Theorem 3.8]{GCMST}.
  For more recent results
 in the   standard case,
also involving the Gaussian conical square function, we refer to \cite{K1,K2}; see also \cite{Wr1, Wr2}, where the author investigates  multiplier theorems for systems of
Ornstein--Uhlenbeck operators. Overviews of this topic can be found in Urbina's monograph \cite[Chapter 6]{Urbina},   Bogachev's survey \cite{Bogachev} and the references therein.

 In the general case, when
 $\mathcal L$ is
 given by \eqref{def-L-OU},
  the strong $L^p(\gamma_\infty)$ boundedness
 of $m(\mathcal L)$ follows for $1<p<\infty$  from \cite[Prop. 3.8]{Carbonaro-Oliver}.
In the present paper, we  consider the endpoint  case $p=1$, where the strong boundedness does not hold.

Our main result is the following.
\begin{theorem}\label{weaktype1}
If the function $m$ is of Laplace transform type, then the multiplier operator
$m(\mathcal L)$
associated to a general Ornstein--Uhlenbeck operator $\mathcal L$
  is of weak type $(1, 1)$ with respect to the invariant measure $d\gamma_\infty$.
 \end{theorem}
Thus we shall prove the inequality
\begin{equation} \label{thesis-mixed-Di}
\gamma_\infty
\{x\in\R^n : m(\mathcal L)\,
f(x) > C\alpha\} \le \frac{C}\alpha\,\|f\|_{L^1( \gamma_\infty)},\qquad \text{ $\alpha>0$,}
\end{equation}
 for  all  functions $f\in L^1 (\gamma_\infty)$,
with $C=C(n,Q,B)$.
     Our theorem extends Theorem~3.8 in  \cite{GCMST} to the framework of a general,   not necessarily normal, Ornstein--Uhlenbeck operator.


  In this paper, 
   we do not deal with
   holomorphic H\"ormander-type functional calculus.  
   For results in that context
the reader is referred in particular to  \cite{Carbonaro-Oliver} and to   \cite{HP}.
The literature in this field is vast;   good bibliographies are given in   
  \cite{Carbonaro-Oliver, HP, P}.

A  careful study of other
 exponential 
 integral inequalities 
 in the Gaussian framework may be found in \cite{CMP1, CMP2, CMP3}.
Finally, it is worth mentioning that
several interesting results concerning other issues of harmonic analysis
 in a nonsymmetric Ornstein--Uhlenbeck context, such as square functions, maximal operators and variational bounds,
  have recently appeared in \cite{Almeida, Almeida2}.


\medskip

What follows next is a description of the structure of the paper, which also gives
 a plan of the proof of Theorem \ref{weaktype1}.

In Section \ref{preliminaries}, we introduce some terminology and
recall from the authors' earlier papers \cite{CCS1, CCS2, CCS3}
a few estimates which are essential in our approach.
Section \ref{def} gives  a rigorous definition of  the multiplier operator,
and in Subsection \ref{defspl} we split this operator by splitting the integrals in \eqref{defmlambda} and
\eqref{heur}
into parts taken over $t<1$ and $t>1$.
Then in Section            \ref{derKt-section}  the time derivative  $\partial_t{K_t} $
of the Mehler kernel is computed and estimated.
This leads in  Section  \ref{parts} to some estimates for the kernels of the different parts of the operator.
There we also introduce some  technical simplifications that will reduce the complexity of the proof of Theorem~\ref{weaktype1}; further reductions will be presented
in Subsection \ref{simply}.
This proof is given in the remaining sections, in the following way.

The operator part with $t>1$ is dealt with in Section  \ref{t-large-section}.
The  part corresponding to $t<1$ is further split into a local and a global part in Section \ref{localization}, and several related estimates are given. Section \ref{local region} contains the proof for the local part
with standard
Calder\'on-Zygmund techniques.
The remaining, global part is more delicate. For its kernel we will have a bound
\begin{equation*}
\int_0^{1}
|\partial_t K_t (x,u)| \,dt \le \sum \left| \int \partial_t K_t (x,u) \,dt \right|,
\end{equation*}
where the integrals in the sum are taken between consecutive zeros of $\partial_t K_t$.
Therefore, we will need an estimate of
the number of zeros of $\partial_t K_t(x,u)$ as $t$ runs through the interval
  $(0, 1]$. This number turns out to be controlled by a constant depending only on  $n$ and $B$,
as verified in Section \ref{number-of-zeros}.
We can then complete the proof of   the weak type $(1,1)$
 in Section \ref{est-tsmall-glob-section}.
There we also need an enhanced version of the Ornstein--Uhlenbeck maximal operator theorem from [CCS2, Theorem 1.1].  Its proof is given
 in the Appendix (Section \ref{Appendix}).

\medskip

We will write $C<\infty$ and $c>0$ for various constants, all of which  depend only on $n$, $Q$  and $B$, unless otherwise explicitly stated. If $a$ and $b$ are positive quantities, $a \lesssim b$ or
equivalently  $b \gtrsim a$ means  $a \le C b$. When $a \lesssim b$ and also $b \lesssim a$,
we write  $a \simeq b$.
By $\N$ we denote the set of all nonnegative integers,
and $B(x,r)$ is the open ball with center $x$ and radius $r$.
 If $A$ is  a real $n\times n$ matrix, we write $\|A\|$
for its operator norm on $\R^n$
with the Euclidean norm $|\cdot|$\,.
 We will adopt the  dot notation for differentiation with respect to the time variable $t$,
  writing  $ \dot K_t = \partial_t K_t$.
\medskip

The authors  would like to thank Andrea Carbonaro for several helpful discussions.

\section{Preliminaries}\label{preliminaries}

In this section we collect some
results from \cite{CCS1, CCS2, CCS3} related to the Mehler kernel
of a general Ornstein--Uhlenbeck  semigroup.
\subsection{Some matrices and estimates}
~~

\vskip1pt

In terms of the two  real $n\times n$ matrices $Q$ and $B$  introduced in Section 1,
we define for $t\in (0,+\infty]$ the matrix
\begin{equation}\label{defQt}
Q_t=\int_0^t e^{sB}\,Q\,e^{sB^*}\,ds.
\end{equation}
Since $Q$ is real, symmetric and positive definite
and the  eigenvalues of $B$ have negative real parts, this integral  is  convergent  and
the matrix $Q_t$  is symmetric and positive definite and thus   invertible, for all $0<t \le \infty$.

It will be convenient to write
\[
|x|_Q = | Q_\infty^{-1/2}\,x|,  \qquad x \in \R^n,
\]
which is a norm on $\R^n$, and   $|x|_Q \simeq |x|$.
Further, we let $R(x)$ denote the (positive definite) quadratic form
\begin{equation}\label{def:R_normaQ}
R(x) = \frac12\, |x|_Q^2 ={\frac12 \left\langle Q_\infty^{-1}\,x ,x  \right\rangle}, \qquad\text{$x\in\R^n$}.
\end{equation}
The invariant measure is given by
\begin{equation*}
 d\gamma_\infty(x) =  (2\pi)^{-n/2}  \big( \det Q_\infty\big)^{-1/2}\exp(-R(x))\,dx.
\end{equation*}
Notice that $d\gamma_\infty$ is  normalized.

We will also use the one-parameter group of matrices
\begin{align}\label{defDt}
D_t = Q_\infty\, e^{-tB^*}\, Q_\infty^{-1}, \qquad t\in\mathbb R,
\end{align}
introduced in  \cite{CCS2}.
From \cite[formula (2.3) and Lemma 2.1]{CCS2} we know that
\begin{align}\label{Dt_again}
D_t = \left(Q_t^{-1} - Q_\infty^{-1}\right)^{-1}\, Q_t^{-1} \,e^{tB}, \qquad t>0.
\end{align}
and
\begin{align}\label{defDttt'}
D_t = e^{tB} + Q_t\, e^{-tB^*}\,Q_\infty^{-1}, \qquad t>0.
\end{align}
By means of a Jordan decomposition of $B^*$,  the following estimates were proved  in
\cite[Lemma 3.1]{CCS2}
\begin{equation}\label{est:2-eBs-v}
  e^{ct}\,|x| \,\lesssim \,|D_t\, x| \, \lesssim  \, e^{Ct}\, |x|
  \qquad
\text{ and }
\qquad
  e^{-Ct}\,|x|\, \lesssim\, |D_{-t}\, x| \, \lesssim \,  e^{-ct}\, |x|,
\end{equation}
holding for $t>0$ and all  $x\in \R^n$.
The same bounds are true  with  $ D_{t }$ replaced by  $e^{-tB}$ or $e^{-tB^*}$;  in particular,
\begin{equation}
\label{est:2-esBs-v}
  e^{ct}\,|x|\, \lesssim \,|e^{-tB}\, x| \, \lesssim \,  e^{Ct}\, |x|
  \qquad
\text{ and }
\qquad
  e^{-Ct}\,|x|\, \lesssim\, |e^{tB}\, x| \, \lesssim  \, e^{-ct}\, |x|
\end{equation}
for  $t>0$ and   $x\in \R^n$.

From these inequalities one  deduces (see  \cite[Lemma 3.2]{CCS2})
\begin{align}
&\| Q_t^{-1}\|\simeq (\min (1,t))^{-1},\label{stimaQt-1}
\\
&\|Q_t^{-1}-Q_\infty^{-1}\|\lesssim {t}^{-1}\,{e^{-ct}}\label{stime Q}.
\end{align}
Finally, we recall the following lemma, proved in \cite[Lemma 2.3]{CCS3}.
\begin{lemma} \label{differ}
  Let $ x \in \R^n$ and $|t| \le 1$. Then
  \begin{equation*}
    |x- D_t \,x| \simeq |t|\, |x|.
  \end{equation*}
\end{lemma}

\subsection{Spectrum and generalized eigenspaces of $\mathcal L$}

\vskip1pt

Let $\lambda_1,\dots, \lambda_r$ be the eigenvalues of $B$.
   It is known that the spectrum of $\mathcal L$
in $L^p(\gamma_\infty)$, $1<p<\infty$, is
\begin{equation}\label{spL}
    \left\{-\sum_{i=1}^r  n_i\,\lambda_i: n_i \in \N,\; i =1,\dots,r \right\} \subset \{z \in \C:\: \Re z>0 \} \cup \{0\},
\end{equation}
see \cite[Theorem 3.1]{MPP}.

Each point $\lambda$ in this set is an eigenvalue of  $\mathcal L$. The
corresponding generalized eigenfunctions, i.e., the functions annihilated by
  $(\mathcal L - \lambda)^k$ for some $k \in \N$, are polynomials, see  \cite[Theorem  9.3.20]{Lorenzi}.
For each $\lambda$ they form a finite-dimensional space, and these generalized
eigenspaces together span a dense subspace of $L^2(\gamma_\infty)$.
In particular, $0$ is an eigenvalue of $\mathcal L$. The corresponding
eigenspace, which we denote by $\mathcal E_0$, is of dimension $1$ and consists of the constant functions. As shown in \cite[Lemma 2.1]{CCS4},
this eigenspace is orthogonal to all other generalized eigenfunctions of $\mathcal L$.
We denote by  $L^2_0(\gamma_\infty)$ the orthogonal complement of $\mathcal E_0$ in  $L^2(\gamma_\infty)$.

\vskip15pt

\subsection{The Mehler kernel}
 For $x,u\in\R^n$ and $t>0$ the Mehler kernel $K_t$   is given by
(see \cite[formula (2.6)]{CCS2})
\begin{align}\label{defKRt}
K_t (x,u)
=
\Big(
\frac{\det \, Q_\infty}{\det \, Q_t}
\Big)^{{1}/{2} }\,
e^{R(x)}\,
\exp \Big[
{-\frac12
\left\langle (
Q_t^{-1}-Q_\infty^{-1}) (u-D_t \,x) \,,\, u-D_t\, x\right\rangle}\Big].\quad\quad
\end{align}

It is convenient to use this expression  for $K_t$ when $t\leq1$.
But  for $t\ge 1$, we will use the following alternative, which can be obtained from
\cite[first formula in the proof of Proposition 3.3]{CCS2},
\begin{align}\label{defKRt1}
K_t (x,u)
=
\Big(
\frac{\det Q_\infty}{\det  Q_t}
\Big)^{1/2 }
e^{R(x)}
\exp \Big[
{-\frac12
\left\langle Q_t^{-1}
e^{tB}  ( D_{-t}\,u-x) , D_t \,( D_{-t}\,u-x)\right\rangle}\Big].\quad\quad
\end{align}
\medskip

For $0<t\leq 1$ we have the following
estimates, proved in \cite[(2.10)]{CCS3}
\begin{equation}\label{litet}
   \frac{ e^{R( x)}}{t^{n/2}}\exp\left[-C\,\frac{|u-D_t \,x |^2}t\right]
 \,\lesssim\,   K_t(x,u)
\,\lesssim \, \frac{ e^{R( x)}}{t^{n/2}} \exp\left[-c\,\frac{|u-D_t\, x |^2}t\right].
\end{equation}
When $t\geq 1$
one has instead (see \cite[(2.11)]{CCS3})
\begin{equation} \label{tstort}
e^{R(x)}
\exp
\Big[
-C
\left|
D_{-t}\,u- x\right|_Q^2
\Big]
\,\lesssim\,
K_t (x,u)
\,\lesssim\,
e^{R(x)}
\exp
\Big[
-\frac12
\left|
D_{-t}\,u- x\right|_Q^2
\Big].\end{equation}

\subsection{Polar coordinates}
We will use a variant of polar  coordinates first introduced  in \cite{CCS1}.
Fix  $\beta>0$ and consider the ellipsoid
\begin{equation*}
\ellipses_\beta
=\{x\in\R^n:\, R(x)=
\beta\}
\,.\end{equation*}
Any $x\in\R^n,\, x\neq 0$,
can be written uniquely as
\begin{equation}\label{def-coord}
x=D_s \,\tilde x
\,,
\end{equation}
for some $\tilde x\in \ellipses_\beta$
and $s\in\R$. We call $(s, \tilde x)$ the
polar coordinates of $x$.

The Lebesgue measure in $\R^n$ is given in terms of
$(s, \tilde x)$ by
\begin{align}\label{def:leb-meas-pulita}
  dx =
e^{-s\tr B}\, \frac{ |Q^{1/2}\, Q_\infty^{-1} \tilde x |^2}
{2\,| Q_\infty^{-1} \tilde x  |}\,
 dS_\beta( \tilde x)\,ds\,,
\end{align}
where $dS_\beta$ denotes the area measure of $\ellipses_\beta$.
See \cite[Proposition 4.2]{CCS2} for a proof.


      \section{Definition  and splitting of the multiplier operator}\label{def}

\subsection{Definition   of  the multiplier operator}
We use the definition described in Cowling et al.\ \cite[Section 2]{CDMY}, which goes back to McIntosh \cite{M}. The starting-point in  \cite{CDMY} is
 an operator $T$ defined  on a Hilbert (or Banach) space, which will be  $L^2_0(\gamma_\infty)$ in our case. This operator is to be densely defined and one-to-one with dense range,
 and its spectrum must be contained in  a closed sector
 \begin{equation*}
S_\omega = \{z \in \C: |\mathrm{arg} z| \le \omega\}\cup \{ 0\},
 \end{equation*}
for some $\omega \in (0, \pi/2)$.
Further,  the resolvent of $T$ should satisfy the estimate
\begin{equation}\label{resolvent}
   \|(T-zI )^{-1}  \| \le C \,|z|^{-1},\qquad  z \in \C\setminus S_\omega,
\end{equation}
for some constant $C$,
where we refer to the operator norm on  $L^2_0(\gamma_\infty)$.

Therefore, we define the operator $T$ as the restriction of  $\mathcal L$ to $L^2_0(\gamma_\infty)$.
 We will prove Theorem \ref{weaktype1} with $\mathcal L$ replaced by $T$. The theorem then follows, since $\mathcal L$ vanishes on $\mathcal E_0$.

From the preceding section,
it is clear that $T$ has all the  properties required in \cite{CDMY}  mentioned above, except possibly the inequality \eqref{resolvent}.
We shall now verify \eqref{resolvent}.

According to \cite[Theorem 1 and Remark 6]{CFMP1}, there exists an angle $\theta_2 \in (0, \pi/2)$ such that the semigroup $\left(e^{-tT}\right)_{t>0}$ is a contraction on $L^2_0(\gamma_\infty)$ for each $t$ in the sector
$S_{\theta_2}$.
Then  \eqref{resolvent}  follows from some well-known arguments for bounded analytic semigroups (see \cite[Ch. II, Section 4.a]{Engel}).
Anyway, we
give
 a concise proof.

Fix a $\theta \in (0, \theta_2)$; like $\theta_2$ this $\theta$ will only depend on  $n$, $Q$  and $B$.
If $z$ is on the negative real axis,  the contraction property implies
\begin{equation} \label{ray}
   (T-zI )^{-1} = \int_{0}^{+\infty} e^{-t(T-zI)}\,dt = e^{i\theta}\, \int_{0}^{+\infty} e^{-te^{i\theta}T}\, e^{te^{i\theta}zI}\,dt,
\end{equation}
where we moved  the path of integration to the  ray $e^{i\theta}\,\R_+$ in $\C$.
Here we want to let $z = re^{i\varphi}$,
with $r > 0$ and  $\varphi\in (\pi/2 - \theta/2 , \pi] $. Then
\begin{equation*}
0 < \theta/2 < \theta + \varphi -\pi/2 \,  \,\le\, \, \theta + \pi/2 < \pi
\end{equation*}
 and  so
\begin{equation*}
   \Re (t\,e^{i\theta}z) = t\,r\cos(\theta + \varphi) = -t\,r\sin(\theta + \varphi-\pi/2) < -c\,t\,r.
\end{equation*}
For such $z$ the second integral in \eqref{ray} converges, and by analyticity  it equals $e^{-i\theta}\,(T-zI )^{-1}$. Thus
\begin{equation*}
 \|(T-zI )^{-1}\| \le  \int_{0}^{+\infty}   e^{-ctr}\,dt
 \le \frac{C}{|z|},
\end{equation*}
which proves \eqref{resolvent} for $z$ in the upper half-plane, with $\omega = \pi/2 - \theta/2$. To deal with the case when $z$ is in the lower half-plane,
it is enough to take the complex conjugate of the equation  \eqref{ray} and repeat the argument, because  $T$ is real.
We have thus verified  \eqref{resolvent}.

Since 0 is not in the spectrum of $T$, we have the following improvement of
\eqref{resolvent}:
\begin{equation}\label{resolventimp}
   \|(T-zI )^{-1}  \| \le C \,(1 + |z|)^{-1},\qquad  z \in \C\setminus S_\omega.
\end{equation}

\vskip3pt

The function $m$ is of Laplace transform type and thus defined and analytic in the right half-plane.
Moreover, it is bounded  on any sector  $S_\phi$ with $0 < \phi <\pi/2$.
The definition  of $m(T)$ in \cite{CDMY}  goes via a complex integral involving the resolvent of $T$. To make this integral convergent,
we multiply the function $m(z)$ by
$\psi(z) = 1/(1+z^2)$,
following \cite{CDMY}.
With  $\omega \in (0, \pi/2)$ fulfilling  \eqref{resolventimp},
 we  fix a $\nu \in (\omega, \pi/2)$ and let $\Gamma$ be the path
\begin{equation*}
\Gamma(t) = |t|\,e^{i\nu\, \mathrm{sgn}\, t},  \qquad -\infty < t < \infty.
\end{equation*}
Now define
\begin{equation*}
  (\psi m)(T) = \frac 1 {2\pi i} \int_\Gamma \psi(z) m(z)\,(zI-T)^{-1}\,dz,
\end{equation*}
which is a convergent integral because of \eqref{resolventimp}, and let
\begin{equation}  \label{defmt}
m(T) =  \psi(T)^{-1} (\psi m)(T).
\end{equation}

\begin{proposition}\label{restr}
  Let $\lambda \ne 0$ be a  generalized eigenvalue  of $T$ with generalized eigenspace
 $\mathcal E_\lambda$. Then  the restriction to  $\mathcal E_\lambda$ of $m(T)$ (defined above)
 coincides with the restriction to  $\mathcal E_\lambda$ of the integral
   \begin{equation*}
- \int_0^{+\infty}
\varphi (t)\,
\frac{d}{dt}\, e^{ -tT}
\,dt.
\end{equation*}.
\end{proposition}
Notice that this is the integral from
  \eqref{heur}, and that its restriction to the finite-dimensional, $T$-invariant subspace
$\mathcal E_\lambda$ makes perfect sense.  Further,  $m(T)$ is determined  by these restrictions, since
 the $\mathcal E_\lambda$ together span
$L^2_0(\gamma_\infty)$ and $m(T)$ is  bounded on $L^2_0(\gamma_\infty)$,
as proved by  \cite[Lemma 3.7]{Carbonaro-Oliver}.

\begin{proof}
Observe first that $T\big|_{\mathcal E_\lambda} = \lambda I + R_\lambda$, where $R_\lambda$ a nilpotent operator on $\mathcal E_\lambda$.
For $z \in \Bbb C \setminus \{\lambda\}$  this leads to
\begin{align*}
  (zI-T)^{-1}\big|_{\mathcal E_\lambda}\,= \,\left((z-\lambda)I-R_\lambda\right)^{-1}
  = &\, (z-\lambda)^{-1}\left(I - \frac{R_\lambda}{z-\lambda}\right)^{-1} \\
  = & \,\sum_j \frac{1}{(z-\lambda)^{j+1}}\,R_\lambda^j ,
\end{align*}
where the sum is finite. Thus
\begin{align*}
(\psi m)(T)\big|_{\mathcal E_\lambda}\,
=&\: \frac 1 {2\pi i}\, \int_\Gamma  \psi(z)\, m(z)\,   (zI-T)^{-1}\big|_{\mathcal E_\lambda}\,dz \\
=& \: \frac 1 {2\pi i} \sum_j \int_\Gamma \psi(z)\, m(z)\, \frac{1}{(z-\lambda)^{j+1}}\,dz\;R_\lambda^j \\
=& \, \sum_{j} \frac 1 {j!} \, (\psi m)^{(j)}(\lambda)\,R_\lambda^j\\
=& \, \sum_{i,\,k} \frac 1 {i!\,k!}\,\psi^{(i)}(\lambda)\,m^{(k)}(\lambda)\,R_\lambda^{i+k}\\
=& \,  \sum_{i} \frac 1 {i!}\,\psi^{(i)}(\lambda)\,R_\lambda^{i} \:
 \sum_{k} \frac 1 {k!}\,m^{(k)}(\lambda)\,R_\lambda^{k} \\
 =&\: \psi(T)\, \sum_{k} \frac 1 {k!}\,m^{(k)}(\lambda)\:R_\lambda^{k}.
   \end{align*}

From \eqref{defmt} and  \eqref{defmlambda}, we conclude that
\begin{align*}
 m(T)\big|_{\mathcal E_\lambda}\,
& = \sum_{k} \frac 1 {k!}\,m^{(k)}(\lambda)\:R_\lambda^{k} \\
 & = - \sum_{k} \frac 1 {k!}\, \int_0^{+\infty}
\varphi (t)\, \left(\frac{\partial}{\partial \lambda}\right)^{k}
\frac{d}{dt}\, e^{-\lambda t} \,dt \:R_\lambda^{k} \\
 & = - \int_0^{+\infty} \varphi (t)\,\frac{d}{dt}\,
 \sum_{k} \frac 1 {k!}\,\left(\frac{\partial^{k}}{\partial \lambda^{k}}\,e^{-\lambda t}\right)\, dt  \:R_\lambda^{k}.
   \end{align*}

   Here the sum equals
   \begin{align*}
\sum_{k}e^{-\lambda t}\, \frac 1 {k!}\, (-t)^k\, R_\lambda^{k} = e^{-\lambda t}\, e^{- tR_\lambda} = e^{-tT},
   \end{align*}
   and the  proposition follows.
\end{proof}

\subsection{Splitting of the multiplier operator}\label{defspl}

Given  $\varphi\in L^\infty (0, +\infty)$, we will restrict the integral in \eqref{defmlambda} to various intervals. For $\varepsilon > 0$ we let
\begin{equation*}
m_\varepsilon(\lambda) =
-\int_\varepsilon^{+\infty} \varphi (t)\, \frac{d}{dt}\, e^{-\lambda t}\,dt.
\end{equation*}
But replacing $\epsilon$ by $0$ we also define, in a slightly inconsistent way,
\begin{equation*}
m_0(\lambda) =
-\int_0^{1} \varphi (t)\, \frac{d}{dt}\, e^{-\lambda t}\,dt,
\end{equation*}
and observe that
\begin{equation*}
m(T) = m_1(T) + m_0(T).
\end{equation*}
Then \eqref{Mehler-kernel-derivative-mult} hints that $m_\varepsilon(T)$ and $m_0(T)$ should have off-diagonal kernels

given by
\begin{equation} \label{eps}
\mathcal M_{\varepsilon} (x,u)
:=
 -\int_\varepsilon^{+\infty}
\varphi (t)\,
\dot K_t (x,u)
\,dt
\end{equation}
and
\begin{equation}\label{mnoll}
\mathcal M_0 (x,u)
:=
 -\int_0^1
\varphi (t)\,
\dot K_t (x,u)
\,dt.
\end{equation}
As will be verified in Section~\ref{parts},
 $\mathcal M_{\varepsilon}$ is the kernel of $m_\varepsilon(T)$ for any  $\varepsilon > 0$, and we use it   
in Section~\ref{t-large-section} to control $m_1(T)$. But $\mathcal M_0$ is singular and \eqref{mnoll} is problematical on the diagonal $x=u$.
We shall need to consider separately the global and local parts of $m_0(T)$; they will be suitably defined in
Section \ref{local region}.

\section{The time derivative of the Mehler kernel}\label{derKt-section}

We compute the derivative $\dot{K_t}
= \partial_t K_t (x,u)$ 
 and estimate it for small and large $t$.
As a preparation, we  work out the $t$ derivatives
of some of the matrices introduced in the previous section.

\begin{lemma}\label{derivatives}
For all $t>0$ one has
\begin{align}
\dot{Q_t}&=\,
e^{tB}\, Q\, e^{tB^*};
\label{28}\\
\frac{d}{dt}\, Q_t^{-1}&=-Q_t^{-1} \, \dot{Q_t}\, Q_t^{-1}
=\, -Q_t^{-1}\, e^{tB}\, Q\, e^{tB^*}\, Q_t^{-1}; \label{8}\\
\frac{d}{dt}\,
\det Q_t
&= \,\det Q_t\,
\tr (Q_t^{-1} \, \dot{Q_t})       = \det Q_t\, \tr (Q_t^{-1} \, e^{tB} Q e^{tB^*});
\label{Jacobi}\\
\dot D_t&=\, - Q_\infty\, B^*\, e^{-tB^*}\, Q_\infty^{-1}
=\, - Q_\infty\, B^* \, Q_\infty^{-1}\, D_t.
\label{derd-t}
\end{align}
\end{lemma}

\begin{proof}
The equality \eqref{28}
trivially follows from \eqref{defQt}.
To obtain \eqref{8}, one differentiates the equation $Q_t \, Q_t^{-1}
=I$ and applies \eqref{28}.
Since $Q_t$ is nonsingular,  Jacobi's formula
implies \eqref{Jacobi} (see \cite[Fact 10.11.19]{Bernstein}).
Finally, we obtain the two equalities in \eqref{derd-t} from \eqref{defDt}.
\end{proof}

It will be convenient  to have two different expressions for
the $t$ derivative of the Mehler kernel, as follows.

\begin{lemma}\label{derivate-nucleo}
For all $(x,u) \in \mathbb R^n\times \mathbb R^n$ and
$t>0$, we have
\begin{align*}
\dot K_t(x,u) = K_t(x,u) \, N_t(x,u),
\end{align*}
where the function $N_t$ is given by
\begin{align}\label{R}
\notag
N_t (x,u)&=-\frac12\,{\tr
\big(Q_t^{-1} \, e^{tB}\, Q\, e^{tB^*}\big)}
+\frac12\,
\left| Q^{1/2}\, e^{tB^*}\, Q_t^{-1}\,(u-D_t\, x)
\right|^2
\\
&\qquad\qquad\qquad
-
\left\langle Q_\infty \,B^*\, Q_\infty^{-1}\, D_t\, x\,,\,
(Q_t^{-1}-Q_\infty^{-1})\,(u-D_t\, x)\right\rangle,
\end{align}
and also by
\begin{align}\label{P}
\notag
 N_t (x,u)=&
-
\frac12\,
{\tr\left(Q_t^{-1} \, e^{tB}\, Q\, e^{tB^*}\right)}{}
+
\frac12 \, \left|
Q^{1/2}\, e^{tB^*}\, Q_t^{-1}\, e^{tB}\,
(D_{-t}\,u- x)
\right|^2
\\\notag
&
-
\left\langle
Q_t^{-1}\,
B\, e^{tB}\,
(D_{-t}\,u- x)
\,,\,
e^{tB}\,
(D_{-t}\,u- x)
\right \rangle
\\\notag
&
-
\left\langle
Q_t^{-1}\,
e^{tB}\,
Q_\infty\, B^*\,  Q_\infty^{-1}\, D_{-t}\, u\,,\,
e^{tB}\,
(D_{-t}\,u- x)
\right \rangle\notag
\\ &
-
\left\langle
 B^* \, Q_\infty^{-1}\, D_{-t}\, u \,,\, D_{-t }\,u-x
\right\rangle  \notag
\\
=:& \:I_t+II_t(x,u)+III_t(x,u)+IV_t(x,u)+V_t(x,u).
\end{align}
\end{lemma}

\begin{proof}
 Differentiating \eqref{defKRt} with respect to $t$ and applying Lemma \ref{derivatives},
one obtains
\begin{multline*}
\dot K_t (x,u)
= \\
K_t (x,u)
\Big[
-\frac12\,{\tr \big(Q_t^{-1}\, e^{tB}\, Q\, e^{tB^*} \big)}
+\frac12
\left\langle Q_t^{-1}\, e^{tB}\, Q\, e^{tB^*}\,Q_t^{-1} (u-D_t\, x) \,,\, u-D_t\, x\right\rangle \\
-\left\langle (Q_t^{-1}-Q_\infty^{-1})\,Q_\infty \,B^*\, Q_\infty^{-1}\, D_t\, x\,,\,
(u-D_t\, x)\right\rangle
\Big],
\end{multline*}
from which \eqref{R} follows.

Next, we differentiate  \eqref{defKRt1},  applying   \eqref{Jacobi}\ to the first factor, and then use  \eqref{defDttt'} to rewrite the matrix $D_t$ in the exponent. The result will be
\begin{align} \label{second-expr-Kt}  %
\dot K_t (x,u)  \notag
& =  K_t (x,u)
\Big\{
-\frac12\,{\tr \big(Q_t^{-1}\, e^{tB}\, Q\, e^{tB^*}\big)}\\
+\frac{d}{dt} & \Big[
{-\frac12
\left\langle Q_t^{-1}\,
e^{tB}\,  ( D_{-t}\,u-x) \,,\,  (e^{tB} + Q_t\, e^{-tB^*}\,Q_\infty^{-1}) ( D_{-t}\,u-x)\right\rangle}\Big]\Big\}.
\end{align}
The derivative here  consists of two terms, the first term being
     \begin{align*}
     \frac{d}{dt}& \Big[-\frac12 \left\langle Q_t^{-1}\, e^{tB}\, ( D_{-t}\,u-x) \,,\,  e^{tB} \, (D_{-t}\,u-x) \right\rangle\Big] \notag
       \\&= \frac12\,\big\langle  Q_t^{-1} \,e^{tB}\, Q\, e^{tB^*} \,Q_t^{-1}\, e^{tB}\,  ( D_{-t}\,u-x) \,,\,  e^{tB}\,  ( D_{-t}\,u-x)\big\rangle  \notag
      \\&-\, \big\langle Q_t^{-1}\, B\, e^{tB}\,  ( D_{-t}\,u-x) \,,\,  e^{tB}\,  ( D_{-t}\,u-x)\big\rangle \notag
    \\ &-\, \big\langle Q_t^{-1}\, e^{tB}  \,Q_\infty\, B^*\,  Q_\infty^{-1}\,  D_{-t}\,u \,,\,  e^{tB}\,(D_{-t}\,u-x)\big\rangle,                               
     \end{align*}
 where we applied \eqref{8} and    \eqref{derd-t} with $t$ replaced by $-t$. Notice that we have arrived at the terms $II_t$, $III_t$ and $IV_t$ in \eqref{P}.

 In the second term coming from the derivative in \eqref{second-expr-Kt}, we observe some cancellation;
 the term equals
  \begin{equation*}
      \frac{d}{dt} \Big[
{-\frac12
\left\langle
D_{-t}\,u-x \,,\,   Q_\infty^{-1} \,( D_{-t}\,u-x)\right\rangle}\Big] =
-
\big\langle
D_{-t}\,u-x \,,\,   B^* \, Q_\infty^{-1}\, D_{-t}\,u\big)
\big\rangle = V_t(x,u),
  \end{equation*}
  where we used again  \eqref{derd-t}. Summing up, we obtain  \eqref{P}, and the lemma is proved.         \end{proof}

\begin{lemma}\label{lemma-stime-Pt}
 Let $x,u\in\R^n$.
Then for $0< t\leq1$
\begin{align}\label{R1}
|N_t (x,u)|
\lesssim
\frac{1}{t}
+\frac{\left|u-D_t \,x\right|^2}{t^2}
+ |x|\,\frac{|u-D_t\, x|}t
\end{align}
and for $t\geq1$
\begin{align}\label{P1}
|N_t (x,u)|
\lesssim
|D_{-t}\,u- x|\,|D_{-t}\,u|+
e^{-ct}\,
|D_{-t}\,u- x|^2 +e^{-ct}.           
\end{align}
\end{lemma}

\begin{proof}
For $0< t\leq1$,  \eqref{R1} follows from  \eqref{R}, by means of \eqref{stimaQt-1}
and \eqref{stime Q}.

When $t\geq1$ we get,
starting from \eqref{P}
and using \eqref{est:2-esBs-v} and \eqref{stimaQt-1},\
\begin{align*}
|I_t| =
\frac{1}{2} \, \Big| \tr (Q_t^{-1} \, e^{tB}\, Q\, e^{tB^*}) \Big|
\lesssim
e^{-ct}.
\end{align*}
Similarly, we have
\begin{align*}
|II_t(x,u)| &=
\frac12 \, \Big|
Q^{1/2}\, e^{tB^*}\,Q_t^{-1}\,
e^{tB}\,
(D_{-t}\,u- x)
\Big|^2
\lesssim
e^{-ct}\,|D_{-t}\,u- x|^2,
\end{align*}
and also
\begin{align*}
|III_t(x,u)|            
\lesssim
e^{-ct} \,\left|D_{-t}\,u- x\right|^2.
\end{align*}
Proceeding as above, we further obtain                      
\begin{align*}
|IV_t(x,u)|+|V_t(x,u)|
&\lesssim |D_{-t}u-x|\,|D_{-t}u|,
\end{align*}
and \eqref{P1} is proved.
\end{proof}

\section{On the multiplier kernel}\label{parts}

In this section, we estimate some parts of the multiplier kernel and verify their relevance for the corresponding parts of the operator.
We also state some facts that will simplify the proofs to come.

\subsection{Estimates of kernels}
~
Without loss of generality, it will be assumed from now on that
\[\|\varphi\|_\infty \le 1.\]
We first  invoke a lemma from
  \cite[Lemma 5.1 and Remark 5.5]{CCS3}.
  \begin{lemma}\label{preliminary-t-large-not-kernel}
Let $\delta > 0$. For ${\sigma} \in\{1,2,3\}$ and $x,u\in\R^n$, one has
\begin{equation}
 \int_1^{+\infty}\exp \Big({-\delta\left|  D_{-t}\,u- x
\right|^2 }\Big)\big|   D_{-t} \,  u\big|^\sigma\, dt \lesssim
1+|x|^{\sigma-1},\label{stima-5-inf}
\end{equation}
where the implicit constant  may depend on $\delta$, in addition to $n$, $Q$ and $B$.
\end{lemma}

\begin{proposition}\label{Meps}

 \noindent  {\it{(i)}} The integral   \eqref{eps}        defining $\mathcal M_{\varepsilon}$ converges absolutely
  for any $\varepsilon > 0$ and all $x,u \in \R^n$.
  Moreover,
  \begin{equation}\label{M-1}
|\mathcal M_1(x,u)| \lesssim e^{R(x)}, \qquad x,u \in \R^n,
\end{equation}
and for $0 < \varepsilon < 1$
 \begin{equation}\label{M-eps}
|\mathcal M_{\varepsilon}(x,u)| \lesssim  \varepsilon^{-C}\, e^{R(x)}\,(1+|x|), \qquad x,u \in \R^n.
\end{equation}
\noindent {\it{(ii)}} For any $\varepsilon > 0$, any $f \in  L^2_0(\gamma_\infty)$ and a.a.\  $x \in \R^n$,
\begin{equation}\label{m-eps}
m_{\varepsilon}(T)f(x) = \int \mathcal M_{\varepsilon}(x,u)\,f(u)\,d\gamma_\infty(u).
\end{equation}
\end{proposition}

\begin{proof}
Aiming at {\it{(i)}},
we  begin by estimating the kernel $\dot K_t (x,u) = K_t (x,u)N_t (x,u)$. For $1 < t < +\infty$ we use   \eqref{tstort}    and   \eqref{P1}. Then we can neglect the factors $|  D_{-t}\,u- x|$ in $N_t (x,u)$ by also reducing slightly the positive coefficient in front of the same factor in the exponent in \eqref{tstort}.
As a result,
\begin{equation}\label{dotK1}
|\dot K_t (x,u)| \lesssim
 e^{R(x)}\,\exp{\left(-c\,|  D_{-t}\,u- x|^2\right)}\,(|D_{-t}\,u| + e^{-ct}), \qquad t>1.
\end{equation}
Lemma \ref{preliminary-t-large-not-kernel} now implies \eqref{M-1}.

For $0<t<1$ we use instead
  \eqref{litet}    and   \eqref{R1},
  and now we can neglect all powers of $|u-  D_{t}\,x|^2/t$ in  $N_t (x,u)$.
  This leads to
  \begin{equation}\label{dotKeps}
|\dot K_t (x,u)| \lesssim e^{R(x)}\,t^{-n/2}\,\exp[\left(-c\,\frac{|u-  D_{t}\,x|^2}t \right)]\,
\left(t^{-1}+|x| t^{-1/2}\right) \lesssim  e^{R(x)}\,(1+|x|)\,t^{-n/2-1}.
\end{equation}
Integrating  over  $\varepsilon < t < 1$ and combining the result with
  \eqref{M-1}, we arrive at  \eqref{M-eps}. The claimed convergence is now clear, so  {\it{(i)}}  is verified.

\vskip3pt

For item {\it{(ii)}}, we need  the following lemma.

\begin{lemma}\label{derint}
  Let $f \in  L^2_0(\gamma_\infty)$ and $x \in \R^n$. Then for any $t>0$
  \begin{equation}\label{intder}
  \partial_t\,\int K_t(x,u)\,f(u)\,d\gamma_\infty(u) = \int \dot K_t(x,u)\,f(u)\,d\gamma_\infty(u).
  \end{equation}
\end{lemma}
\begin{proof}
Given  $t>0$, we replace $t$ by $\tau$
in the right-hand side of \eqref{intder},  and then integrate $d\tau$ over the interval
$t_0<\tau<t$, for some fixed $t_0 \in (0,t)$. If we can swap the order of integration in the resulting double integral, we will get
\begin{align*}
\int_{t_0}^{t} \int \dot K_\tau(x,u)\,f(u)\,d\gamma_\infty(u)\,d\tau = \int (K_t(x,u) -K_{t_0}(x,u)))\,f(u)\,d\gamma_\infty(u).
\end{align*}
Differentiating this equation with respect to $t$, we then obtain the lemma.

To see that Fubini's theorem justifies this swap, it is enough to verify that
\begin{align} \label{fubini}
 \sup_{u \in \mathbb{R}^n}  \int_{\varepsilon}^{\infty}  |\dot K_\tau(x,u)|\,d\tau < \infty
\end{align}
for each $x$ and each $\varepsilon>0$. But this follows from  \eqref{dotKeps} for the integral over $\varepsilon < \tau < 1$ and from \eqref{dotK1} and Lemma \ref{preliminary-t-large-not-kernel}  for that over $\tau > 1$.
The lemma is proved.
\end{proof}

To verify item {\it{(ii)}} in the proposition, we  observe that
    because of \eqref{M-eps},
    the right-hand side of \eqref{m-eps} defines for each $x$ a functional on  $L^2_0(\gamma_\infty)$, whose norm is locally uniformly bounded for $x \in \R^n$. Further, the operator $m_{\varepsilon}(T)$ is bounded on the same space (see \cite[Lemma 3.7]{Carbonaro-Oliver}). Since the generalized eigenspaces $\mathcal E_\lambda$ together span $L^2_0(\gamma_\infty)$,
      it is enough to verify \eqref{m-eps} on each $\mathcal E_\lambda$.

So let  $f \in \mathcal E_\lambda$ for some $\lambda$. Since $e^{-tT}\,f(x) = \int K_t(x,u)f(u)\,d\gamma_\infty(u)$,
Proposition~\ref{restr} and  Lemma~\ref{derint} imply
 \begin{align*}
 m_\varepsilon(T)f(x) &= -\int_\varepsilon^\infty  \varphi(t)\,\partial_t \int K_t(x,u)f(u)\,d\gamma_\infty(u)\,dt\\
   &= -\int_\varepsilon^\infty  \varphi(t) \int \dot K_t(x,u)f(u)\,d\gamma_\infty(u)\,dt.
  \end{align*}
  Switching the order of integration,
   again by means of \eqref{fubini},
    we conclude the proof of {\it{(ii)}}.
\end{proof}

\begin{proposition}\label{conv0'}
 \noindent  {\it{(i)}}   The integral   \eqref{mnoll}        defining $\mathcal M_0$ converges for all $x\ne u$, and
  \begin{equation}\label{M-00'}
\left|\mathcal M_0(x,u) \right|\lesssim e^{R(x)} \,(1+|x|)^C\, |x-u|^{-C},
      \qquad                      x \ne u.
\end{equation}
 \noindent  {\it{(ii)}}  For any $f \in  L^2_0(\gamma_\infty)$ and a.a.  $x \notin \mathrm{supp}\,f$,
\begin{equation}\label{m-0'}
m_0(T)f(x) = \int \mathcal M_0(x,u)\,f(u)\,d\gamma_\infty(u).
\end{equation}
\end{proposition}

\begin{proof}
Since $\mathcal M_0$  and  $m_0(T)$ only depend on the restriction of $\varphi$ to the interval $(0,1)$, we can assume in this proof that $\varphi$ vanishes for $t \ge 1$.

To verify {\it{(i)}},  consider the first inequality in \eqref{dotKeps}. We have
 $|u-  D_{t}\,x| \ge |u-x|- |x-  D_{t}\,x|$, and
Lemma \ref{differ}  says that $|x-  D_{t}\,x| \lesssim t|x|$. Thus  $|u-  D_{t}\,x| \ge |u-x|/2$ if  $t<c|u-x|/|x|$ for some $c>0$,
and  we conclude that for $0<t < 1\wedge  c|u-x|/|x|$
\begin{equation}\label{ttt}
|\dot K_t (x,u)| \lesssim e^{R(x)}\,\exp{\left(-c\,\frac{|u- x|^2}t\right)}\,(1+|x|)\,t^{-C} \lesssim e^{R(x)}\,(1+|x|)\, |u-x|^{-2C}.
\end{equation}

The right-hand side here also gives a bound for the integral of
$|\dot K_t (x,u)|$ over the interval $0 <  t <1\wedge (c|u-x|/|x|)$.

 For                 
 $c|u-x|/|x| < t < 1$,  notice that the third quantity in \eqref{dotKeps} is no larger than
 $ e^{R(x)}\,(1+|x|)^C\, |x-u|^{-C}$. After integration, this allows us to end the verification of   {\it{(i)}}.

 Here we observe that, since   $\varphi$ is supported in $[0,1]$, the argument just given leads to
  \begin{equation}\label{unif}
\left|\mathcal M_\varepsilon(x,u) \right|\lesssim e^{R(x)} \,(1+|x|)^C\, |x-u|^{-C},
      \qquad                      x \ne u
\end{equation}
and this is uniform in $\varepsilon$.

  \vskip3pt

   To prove  {\it(ii)}, we will let $\varepsilon \to 0$ in \eqref{m-eps}, with  $x \notin \mathrm{supp}\,f$
   and  $f \in  L^2_0(\gamma_\infty)$.   Consider first the right-hand side of \eqref{m-eps}.

   Because of \eqref{ttt}, we see from   \eqref{eps}   and  \eqref{mnoll}
   that, with $\varphi$ supported in $[0,1]$, one has
   $\mathcal M_\varepsilon(x,u) \to  \mathcal M_0(x,u) $ as $\varepsilon \to 0$, for any  $x \ne u$. In the integral
      in the right-hand side of  \eqref{m-eps}, we thus have pointwise convergence, and
     $|f(u)|\,d\gamma_\infty(u)$ is a finite measure.
      The estimate \eqref{unif} allows us to apply  bounded convergence, and conclude that
    \begin{equation*}
    \int \mathcal M_\varepsilon(x,u)f(u)\,d\gamma_\infty(u) \to
   \int \mathcal M_0(x,u)f(u)\,d\gamma_\infty(u), \qquad
   \varepsilon \to 0,
\end{equation*}
for $x \notin \mathrm{supp}\,f$. Moreover, the left-hand integral here is a function of $x$ which stays locally  bounded in the complement of $\mathrm{supp}\,f$, uniformly in
$\varepsilon$. So we also have convergence in the sense of distributions in $\R^n\setminus\text{supp\,$f$}$.

To deal with the left-hand side of  \eqref{m-eps},
we claim that $m_\varepsilon(T)f \to m_0(T)f$ in the sense of distributions in
 $\R^n                    
 $, as  $\varepsilon \to 0$. This will end the proof of   {\it(ii)}.

 With $\nu,\;\:\Gamma$ and $\psi(z) = 1/(1+z^2)$ as in Section 3, we have
 \begin{equation*}
   m_\varepsilon(T) = (1+T^2)\, \frac 1 {2\pi i} \int_\Gamma  \frac {m_\varepsilon(z)}{1+z^2}\,(zI-T)^{-1} \,dz.
\end{equation*}


 To prove the claim, we let  $f \in L^2(\gamma_\infty)$ and take $g \in C_0^\infty(\R^n)         
 $.
 It is enough to verify that
  \begin{equation*}
   \langle m_\varepsilon(T)f,\, g\rangle \to \langle m(T)f,\, g\rangle,
   \quad \quad \varepsilon \to 0,
\end{equation*}
 with the scalar products taken in  $L^2(\gamma_\infty)$. Notice that it does not matter whether we consider the convergence of the functions
 $m_\varepsilon(T)f$  or the measures $m_\varepsilon(T)f\,d\gamma_\infty$.
 We have
 \begin{align} \label{scalar}
  \langle m_\varepsilon(T)f,\, g\rangle =
& \left\langle (1+T^2)\, \frac 1 {2\pi i} \int_\Gamma   \notag \frac{m_\varepsilon(z)}{1+z^2}\, (zI-T)^{-1} f \,dz,\, g\right\rangle \\
 = \,& \frac 1 {2\pi i} \int_\Gamma \frac{m_\varepsilon(z)}{1+z^2}\,
 \left\langle  (zI-T)^{-1}  f\, ,\, (1+(T^*)^2)\,g
 \right\rangle\,dz,
\end{align}
 where $T^*$ is the adjoint of $T$ in  $L^2(\gamma_\infty)$, so that $(1+(T^*)^2)\,g$ is another test function in $C_0^\infty(\R^n)$. Now
 $m_\varepsilon(z) = z\int_{\varepsilon}^{\infty}\varphi(t)\,e^{-tz} \,dt$
 tends to
 $m(z)$ for each nonzero $z \in \Gamma$. For such $z$
 we also have the bound  $|m_\varepsilon(z)| \le \| \varphi\|_\infty\,
 |z|/\Re z \lesssim  1$.
 In the last integral in \eqref{scalar}, the integrand thus converges pointwise, and it is also dominated by constant times
 \begin{align*}                
  \frac{   1}{1+|z|^2}\: \|  (zI-T)^{-1} f\|_{L^2(\gamma_\infty)}\: \|(1+(T^*)^2)\,g\|_{L^2(\gamma_\infty)},
\end{align*}
which is integrable along $\Gamma$ because of \eqref{resolventimp}.
The dominated convergence theorem now implies the claim and completes the proof of Proposition \ref{conv0'}.
\end{proof}

\subsection{Simplifications}\label{Simplifications}
The preceding estimates
allow some preliminary observations that will simplify the proof of Theorem~\ref{weaktype1}.

   In \eqref{thesis-mixed-Di} we take
$f \ge 0$ such that
 $\|f\|_{L^1( \misgaussk)}=1$. We can then 
 assume that $\alpha$ in the same estimate is large, in particular $\alpha >2$, since $d\gamma_\infty$ is finite.

Further, we
 can focus mainly  on points  $x$
 in the ellipsoidal annulus
 \begin{equation}\label{crown}
{\mathcal C_\alpha}=\left\{
x \in\R^n:\, \frac12  \log \alpha\le
R(x)
\le 2  \log \alpha
\right\}.
\end{equation}
\smallskip
To justify this,
 we will follow closely the arguments in \cite[Section 6]{CCS3}. The first observation is that the
 set of points $x$ for which  $R( x)> 2 \log \alpha$ can be neglected,
because its  $d\gamma_\infty$ measure is no larger than $C/\alpha$.

The next proposition deals with the remaining part of the complement of $\mathcal C_\alpha$ when $t>1$ and  follows immediately from \eqref{M-1}.

   \begin{proposition}\label{inner0}
Let $x\in \R^n$ satisfy
   $R(x) < \frac12 \log \alpha$, where $\alpha > 2$.
   Then for all $u\in \R^n$
 \[
\big|\mathcal M_1 (x,u)\big|
\lesssim { \alpha}.
\]
\end{proposition}


Further simplifications will be introduced in Subsection \ref{simply}.

 \section{The  weak type $(1,1)$ for  large $t$ }\label{t-large-section}
\begin{proposition}\label{stima minfty}
For any
     $f\in L^1 (\gamma_\infty)$ such that
$\|f\|_{L^1( \gamma_\infty)}=1$  and any $\alpha>2$,
\begin{equation*}
  \qquad \qquad \gamma_\infty
\left\{x  \in\mathcal C_\alpha :
 |m_1(T)   f(x)|
 > \alpha\right\} \lesssim \frac{1}{\alpha\sqrt{\log \alpha}}.
\end{equation*}
In particular,
the  operator
 $m_1(T)$
  is of weak type $(1, 1)$ with respect to the invariant measure $d\gamma_\infty$.
  \end{proposition}

The estimate in  this   proposition      means  that for large $\alpha$ one has a slightly stronger estimate than the classical
weak type $(1,1)$ bound.
This phenomenon was  already observed for the Ornstein--Uhlenbeck    maximal operator  in \cite[Section 6]{CCS2},
for the  first-order Riesz transforms in \cite[Proposition 7.1]{CCS3})
and for the variation operator of the Ornstein–Uhlenbeck semigroup in dimension one
in \cite[Proposition 3.1]{CCS6}.

\begin{proof}
We will first use our polar coordinates to deduce a sharper version of the estimate \eqref{M-1} in Proposition \ref{Meps}{\it{(i)}}.
 If  $x  \in\mathcal C_\alpha$ and $u \ne 0$, we can write
 $x = D_s \,\tilde x$ and
$u = D_\sigma\, \tilde u$ with  $\tilde x ,\,\tilde u\in E_{(\log\alpha)/2}$ and $s \geq 0$, $\sigma\in\R$.

Let $t \ge 1$.
Applying \cite[Lemma 4.3 {\it{(i)}}]{CCS2},
we obtain
\begin{align*}
|D_{-t}\,u-x| &= |D_{\sigma-t}\,\tilde u-D_s\, \tilde x|
\gtrsim |\tilde x-\tilde u|.
\end{align*}

Thus \eqref{tstort} implies
\begin{align*}
K_t(x,u)
\lesssim
e^{R(x)}\,
\exp
\Big(
-c\,
 |\tilde x-\tilde u|^2\Big)
\,\exp
\Big(
-c\,
\left|
D_{-t}\,u- x\right|^2
\Big),
\end{align*}
for some $c$.

Using this estimate instead of
\eqref{tstort}, one can follow  the proof of \eqref{M-1}  with an extra factor $\exp \left(- c\,|\tilde x-\tilde u|^2\right)$.
The result will be
\begin{equation}\label{vv}
\big|\mathcal M_1 (x,u)\big| \lesssim  e^{R( x)} \,
 \exp\big(- c\,\big|\tilde x-\tilde u\big|^2\big), \qquad
 x  \in\mathcal C_\alpha.
\end{equation}

 We can now finish the proof of  Proposition \ref{stima minfty} by means of  the following lemma,
which is the case $\sigma = 1$ of
  \cite[Lemma 7.2]{CCS3}.

  \begin{lemma}\label{bounding-GMAa}
 Let $f\ge 0$ be normalized in $L^1 (\gamma_\infty)$. For $\alpha > 2$
 \[
\gamma_\infty \left\{ x = D_s\,\tilde x \in\mathcal C_\alpha:\:
e^{R(x)}
\int
\exp
 \big(
- c\, \big|
\tilde x-\tilde u\big|^2\big)
\,f(u)\,\misgausskd
 (u)
>
\alpha
\right\}
\lesssim \frac{C}{\alpha\,\sqrt{\log \alpha}}.
\]
                     \end{lemma}
\end{proof}

\section{Localization}\label{localization}

In order to study the  weak type $(1,1)$ for  small $t$,  we need a further splitting of the operator   $m_0(T)$. 
Localization means considering that part of  $m_0(T)f(x)$ which depends on the values of $f$ in a ball of center $x$ and radius essentially  $1/(1 +|x|)$. We start by "filling" $\mathbb R^n$ with balls of this type.

\subsection{Local and global parts}

With $x_0 = 0$, we select a sequence $(x_j)_0^\infty$ of points such that the open balls $B_j = B(x_j, 1/(1 +|x_j|))$,\hskip6pt $j=0,1,\dots$,  are pairwise disjoint, and such that the family $(B_j)_0^\infty$is maximal with respect to this property.

In the sequel, we will often use notations like $3B_j$ in the sense of concentric scaling.

We first verify that the balls $3B_j$,\hskip6pt $j=0,1,\dots$, cover  $\mathbb R^n$.    
To verify that a given point $x$ is in some  $3B_j$, we
 can assume that $|x| \ge 3$, since otherwise $x \in 3B_0$. The maximality property of the  $B_j$ implies that
the ball $B(x, 1/(1 +|x|))$ must intersect some  $B_j$,  necessarily with $j>0$. Then
\begin{equation}\label{x-x_j}
|x-x_j| < \frac 1 {1 +|x|} + \frac 1 {1 +|x_j|},
\end{equation}
 and so
\begin{equation*}
 |x_j| < |x| + \frac 1 {1 +|x|} + \frac 1 {1 +|x_j|} \le |x| + \frac 1 4 + \frac 1 2
 \le |x| + \frac 3 4 \, \frac {|x|} 3 = \frac 5 4\, |x|.
\end{equation*}
Thus $1 +|x_j| < 5(1 +|x|)/4$, and \eqref{x-x_j} implies
\begin{equation*}
|x-x_j| < \frac 5 4\,\frac 1 {1 +|x_j|} + \frac 1 {1 +|x_j|} < \frac 3 {1 +|x_j|},
\end{equation*}
so that $x \in 3B_j$.

We will need another property of these balls: for any $j$ and any $x$
\begin{equation}\label{6B}
x \in 6B_j  \qquad   \Rightarrow \qquad \frac 1 7  < \frac{1 +|x|} {1 +|x_j|} < 7.
\end{equation}
Indeed, if $|x-x_j| < 6/(1 +|x_j|)$ we have
\begin{equation*}
1+|x_j| < 1+ |x| + 6/(1 +|x_j|) \le 7+|x| \le 7(1 +|x|),
\end{equation*}
and the other inequality follows in the same way.

This allows us to show that the balls  $(6B_j)_0^\infty$ have bounded overlap, as follows.
Let $x \in 6B_j$, so that  $x_j \in  B(x, 6/(1 +|x_j|))$. Then
\begin{equation*}
B_j = B(x_j, 1/(1 +|x_j|) \subset B(x, 7/(1 +|x_j|)) \subset B(x, 49/(1 +|x|)),
\end{equation*}
the last step by \eqref{6B}. Comparing Lebesgue measures, we have  $|B_j| \gtrsim (1 +|x|)^{-n}$ and
$|B(x, 49/(1 +|x|))| \lesssim (1 +|x|)^{-n}$. Since the $B_j$ are pairwise disjoint, this can occur for at most a bounded number of $j$. The bounded overlap is verified.

We now introduce functions supported in some of these balls, with which the local part of $m_0(T)$
will be defined. Let first $\rho_i$ for each $i=0,1,\dots$ be a nonnegative, smooth function supported in
$4B_i$ such that $\rho_i= 1$ in  $3B_i$. Its gradient should satisfy $|\nabla  \rho_i(x)| \lesssim 1 +|x|$.
The sum $\sum_0^\infty \rho_i$ is locally finite  and satisfies  $1 \le \sum_0^\infty \rho_i \lesssim 1$ and also
$|\nabla \sum_0^\infty \rho_i(x)| \lesssim 1 +|x|$.

The functions we will use are
\begin{equation*}
r_j = \frac{\rho_j}{\sum_0^\infty \rho_i}.
\end{equation*}
Clearly, $r_j$ is  nonnegative, smooth and supported in $4B_j$, and $\sum_0^\infty r_j = 1$. For the gradient, one has
\begin{equation}\label{nabla1}
  |\nabla  r_j(x)| \lesssim 1 +|x|.
\end{equation}

We will also need smooth functions  $ \widetilde r_j $,\hskip6pt $j=0,1,\dots$,  again
with values in $[0,1]$ but having larger supports. They will satisfy
$\widetilde r_j =1 $ in   $5B_j$     and        
$\text{supp}\:\widetilde r_j\subseteq 6B_j$, and like the $r_j$ they will also verify
\begin{equation}\label{nabla2}
\big|\nabla \widetilde r_j (x)\big|\lesssim 1+|x|.
\end{equation}
We  observe that
 the functions   $\widetilde r_j$ have bounded overlap
and that
\begin{equation} \label{Rconst}
 e^{-R(x)}\simeq e^{-R(x_j)} \qquad \text{for} \qquad
 x\in \text{supp}\:\widetilde r_j,
\end{equation}
 which follows from   \cite[formula (2.9)]{CCS3}  with   $y = x-x_j$  when  $|x|$ is large and  is trivial in the opposite case.

 Let
\begin{equation}\label{def:eta}
 \eta(x,u)=\sum_{j=0}^\infty
\widetilde r_j (x) \,r_j (u),
\end{equation}
and note that for all $x,u\in\R^n$
\begin{equation}\label{eta_ineq}
0 \le \eta(x,u)\le 1.                    
\end{equation}

The following two lemmas express that $\eta(x,u)$ indicates how close the points $x$ and $u$ are to each other.

\begin{lemma}\label{lemma1}
 ($i$) If $\eta(x,u) > 0$ for some points  $x$ and $u$, then $|x-u| \lesssim \frac 1 {1 +|x|}$.\\
 ($ii$)          
 For any points with $x\neq u$
\begin{equation}\label{nablaeta}
  \big|\nabla_x\,\eta (x,u)\big|+\big|\nabla_u\,\eta (x,u)\big|
  \lesssim   {|x-u|^{-1}}.
\end{equation}
\end{lemma}
\begin{proof}
  ($i$)  Since  $\eta(x,u) > 0$, there exists a $j$ such that $\widetilde r_j (x) > 0$ and  $ r_j (u) > 0$.
  Thus $x \in 6B_j$ so that $1+|x| \simeq 1+|x_j|$ by \eqref{6B}, and  $u \in 4B_j$.
  We get
  \begin{equation*}
 |x-u| \le |x-x_j|  +  |u-x_j|    < \frac 6{1 +|x_j|} + \frac  4{1 +|x_j|}\lesssim  \frac 1{1 +|x|}.
\end{equation*}
 ($ii$) We can assume that $(x,u)$ is in the support of $\eta$, and by continuity the conclusion of  ($i$) still holds.
 The inequality for $\nabla_x\,\eta$ follows from  \eqref{nabla2} and ($i$). For  $\nabla_u\,\eta$ we apply
 \eqref{6B} to  $x$ and $u$ with $j$ as in  ($i$), to get $1 +|x| \simeq 1 +|u|$. Now we can use  \eqref{nabla1} and ($i$).
\end{proof}

\begin{lemma} \label{lemma2}
  If  $x$ and $u$ are two points with $|x-u| \le   \frac 13 \, \frac 1{1 +|x|}$, then
   \begin{equation} \label{xuclose}
 r_j (u) > 0 \qquad \Rightarrow  \qquad \widetilde r_j (x) = 1,  \qquad j=0,1,\dots,
  \end{equation}
  and $\eta(x,u) = 1$.
\end{lemma}

\begin{proof}
The last conclusion is a consequence of \eqref{xuclose} and
  the definition of $\eta$.
  The implication \eqref{xuclose} follows from
    \begin{equation*}
u \in 4B_j  \qquad \Rightarrow  \qquad x \in 5B_j,
  \end{equation*}
  or equivalently,
   \begin{equation*}
  |u-x_j|    <  \frac  4{1 +|x_j|}  \qquad \Rightarrow  \qquad       |x-x_j|       <  \frac  5{1 +|x_j|}.              \end{equation*}
  But
     \begin{equation*}
|x-x_j| \le |x-u| +  |u-x_j|    < \frac 13 \, \frac 1{1 +|x|} + \frac  4{1 +|x_j|},
\end{equation*}
So it is enough to show that
    \begin{equation*}
 \frac 13 \, \frac 1{1 +|x|} \le \frac 1 {1 +|x_j|},
\end{equation*}
which is obvious if $|x_j| \le 2$. In case $|x_j| > 2$, we have
    \begin{equation*}
  1     +|x_j| \le 1+   |u-x_j| + |x-u| + |x| \le 1 +  \frac  4{1 +|x_j|} + \frac 13 + |x| < 1 +  \frac  43 + \frac 13 + |x|< 3(1 +|x|).
\end{equation*}

\end{proof}

  We now split the multiplier operator   $m_0(T)$  in a local and a global part.
  The local part  is
defined by
\begin{equation}\label{defdefmloc}
m_{0}(T)^{{\rm{loc}}} f(x)
 :=\sum_{j=0}^\infty \widetilde r_j (x)\,
  m_{0}(T)\left(f r_j\right)(x),  \qquad f \in L^1(\gamma_\infty).
\end{equation}
This sum is locally finite and so well defined.

It was proved in Proposition
\ref{conv0'}$(ii)$
that the  off-diagonal kernel of $m_0(T)$ is   $\mathcal M_{0}(x,u) $.
To find that of $m_{0}(T)^{{\rm{loc}}} $, take
$f\in L^1(\gamma_\infty)$.          
For almost all points $x\notin\text{ supp }f$, thus not in
 $\text{ supp }fr_j$ for any $j$, we have
\begin{equation*}
m_{0}(T)^{{\rm{loc}}} f(x)
 =\sum_{j=0}^\infty \widetilde r_j (x)\int \mathcal M_{0}(x,u)
 f(u)\, r_j(u)\,
 d\gamma_\infty (u),\end{equation*}
 where the sum   is again locally finite.
As a consequence,
\begin{equation*}
m_{0}(T)^{{\rm{loc}}} f(x)
 =
  \int \mathcal M_{0}(x,u) \,
 \eta (x,u)\,f(u)\,
 d\gamma_\infty (u),\end{equation*}
 and so the off-diagonal kernel of $m_{0}(T)^{{\rm{loc}}}$
is
\begin{equation}\label{eq:moloc}
{ \mathcal M}_{0}^{{\rm{loc}}}(x,u)=\mathcal M_{0}(x,u) \,
 \eta (x,u).
\end{equation}

We also define
\begin{equation*}
  m_{0}(T)^{{\rm{glob}}}=
   m_{0}(T) -
m_{0}(T)^{{\rm{loc}}}.
\end{equation*}
Its
off-diagonal kernel is
\begin{equation}\label{eq:moglob}
{ \mathcal M}_{0}^{{\rm{glob}}}(x,u)=
\mathcal M_{0}(x,u) \,(1-
 \eta (x,u)).
\end{equation}
Moreover, the next lemma says that          $m_{0}(T)^{{\rm{glob}}}$
is completely given by this kernel.

 \begin{lemma} \label{lemma3}
  For any  $f\in L^1(\gamma_\infty)$ and a.a. $x$, one has
   \begin{align}\label{globoper}
  m_{0}(T)^{{\rm{glob}}}f(x)=
  \int \mathcal M_{0}(x,u) \,
(1- \eta (x,u))\,f(u)\,
 d\gamma_\infty (u).
 \end{align}
 \end{lemma}

\begin{proof}
 Our strategy will be to take a point $\xi \in \mathbb R^n$ and verify \eqref{globoper} a.e. in
 the ball
 $$
 B := B\left(\xi, \frac 18 \, \frac 1{1 +|\xi|}\right).
 $$
 If $x,\: u \in B$ then
 \begin{equation*}
  |x-u| < \frac 14 \, \frac 1{1 +|\xi|},
\end{equation*}
and we have
 \begin{equation*}
1+|x| \le 1+|\xi| +    |x-\xi| <     1+|\xi| +        \frac 18 \, \frac 1{1 +|\xi|} \le  \frac 98\, (1+|\xi|).
\end{equation*}
 Combining these two estimates, we get
  \begin{equation}\label{x-u}
 |x-u| < \frac 14 \, \frac 98 \,\frac 1{1 +|x|} < \frac 13 \,\frac 1{1 +|x|}.
\end{equation}
Thus Lemma \ref{lemma2} applies. We compute $m_{0}(T)^{{\rm{loc}}}(f\chi_B)(x)$ for $x \in B$, by means of the definition
 \eqref{defdefmloc} of  $m_{0}(T)^{{\rm{loc}}}$. If for some $j$, the product  $f\chi_B\,r_j$ does not vanish identically, there exists a point $u \in B$ with $r_j(u) > 0$. Then  Lemma \ref{lemma2} says that  $\widetilde r_j(x) = 1$, and it follows that for $x \in B$
 \begin{align*}
   m_{0}(T)^{{\rm{loc}}}(f\chi_B)(x) = \sum_j m_{0}(T)(f\chi_B\,r_j)(x)
 =\,& \, m_{0}(T)^{{\rm{loc}}}\left(\sum_j
  f \chi_B\,r_j\right)(x) \\
   =\,& \, m_{0}(T)(f\chi_B)(x),
 \end{align*}
  since in the sums here, only a finite number of terms are nonzero
 The equality obtained implies that $m_{0}(T)^{{\rm{glob}}}(f\chi_B)(x) = 0$.

 When we next apply the operator $m_{0}(T)^{{\rm{glob}}}$
 to $f(1-\chi_B)$, we can use the off-diagonal kernel in \eqref{eq:moglob}. As a result, we have for a.a. $x \in B$
 \begin{multline*}
  m_{0}(T)^{{\rm{glob}}}f(x)=
  m_{0}(T)^{{\rm{glob}}} \big( f(1-\chi_B)\big)(x) = \\
   \int \mathcal M_{0}(x, u) \,
(1- \eta (x,u))\,f(u)\, (1-\chi_B(u))
\, d\gamma_\infty (u) =
 \\
    \int \mathcal M_{0}(x,u) \,
(1- \eta (x,u))\,f(u)\,
 d\gamma_\infty (u)
 -  \int \mathcal M_{0}(x,u) \,
(1- \eta (x,u))\,f(u)\, \chi_B(u)
\, d\gamma_\infty (u).
 \end{multline*}
 The last integral  is 0, since here $\eta (x,u) = 1$ in view of \eqref{x-u} and
   Lemma \ref{lemma2}. We have verified \eqref{globoper} a.e. in the ball $B$ and thus almost everywhere.
 \end{proof}

     \subsection{Further simplifications}\label{simply}


The following proposition is complementary to  Proposition \ref{inner0} and deals  with the interior of  $\mathcal C_\alpha$
 when $t<1$.
   \begin{proposition}\label{inner}
Let $x\in \R^n$ satisfy
   $R(x) < \frac12 \log \alpha$, where $\alpha > 2$.
   Then for all $u\in \R^n$
 \[
 \mathcal M_0^{\rm{glob}}(x,u)   
\lesssim { \alpha}.
\]
\end{proposition}

\begin{proof}
If $\mathcal M_0^{\rm{glob}}(x,u) \ne 0$ then  $\eta(x,u)<1$, and Lemma \ref{lemma2} implies that
$ |x-u| \gtrsim \frac1{1+|x|_Q}$.
From  \eqref{M-00'} we then obtain
 \begin{equation*}
   \mathcal M_0^{\rm{glob}}(x,u) \lesssim e^{R(x)}\, (1+| x|)^C \lesssim \alpha,
 \end{equation*}
    and the proposition is verified.
 \end{proof}

 Thus we need to take  the region $\{x:\, R(x) < \frac12  \log \alpha\}$ into account only when considering $m_0(T)^{\rm{loc}}$.

\medskip

\section{The local region}\label{local region}

In this section we shall prove the weak type (1,1) of the operator
$m_{0}(T)^{{\rm{loc}}}$.

In order to apply   Calder\'on--Zygmund theory, we pass to Lebesgue measure.
Set
\begin{equation}  \label{p_Tdef}
 q_T f(x)=
 e^{-R(x)}\,
m_{0}(T)^{{\rm{loc}}}\big ( f(\cdot)\, e^{R(\cdot)}\,\big)(x).
\end{equation}

The relationship between $q_T$ and   $m_{0}(T)^{{\rm{loc}}}$  is clarified by the following result.

\begin{proposition}\label{prop81}
 If  $q_T$ is of   weak type $(1,1)$ with respect to Lebesgue measure, then $m_{0}(T)^{{\rm{loc}}}$  is of   weak type $(1,1)$ with respect to the invariant measure.
\end{proposition}

\begin{proof}
By \eqref{defdefmloc} and the bounded overlap of the $\widetilde r_j$, we have
\begin{align*}
\|m_{0}(T)^{{\rm{loc}}} f\|_{L^{1,\infty}(\gamma_\infty)} &=
\left\|\sum_{j=0}^\infty \widetilde r_j
  m_{0}(T)\left(f r_j\right)
\right\|_{L^{1,\infty}(\gamma_\infty)} \\  %
&\lesssim
\sum_{j=0}^\infty
\left\|\widetilde r_j \,
  m_{0}(T)\left(f r_j\right)
\right\|_{L^{1,\infty}(\gamma_\infty)}.
\end{align*}
 The last sum may be rewritten as
\begin{align*}
&\sum_{j=0}^\infty
\left\|
e^{R(x)}\,
e^{-R(x)}\,
  \widetilde r_j (x)\,
  m_{0}(T) \left(f r_j
\right)(x)
\right\|_{L^{1,\infty}(\gamma_\infty)}\\
=&
\sum_{j=0}^\infty
\left\|
e^{R(x)}\,\widetilde r_j(x)\,
q_T \left(f \,r_j \, e^{-R(\cdot)}
\right)(x)
\right\|_{L^{1,\infty}(\gamma_\infty)}\\
\simeq &
 \sum_{j=0}^\infty
e^{R(x_j)}\,\left\| \widetilde r_j\,
q_T \left(f \,r_j\,  e^{-R(\cdot)}
\right)
\right\|_{L^{1,\infty}(\gamma_\infty)}
\simeq 
\sum_{j=0}^\infty
\left\|\widetilde r_j\,
q_T \left(f\, r_j \, e^{-R(\cdot)}
\right)
\right\|_{L^{1,\infty}(dx)};
\end{align*}
in the last two steps we used \eqref{Rconst}.
The hypothesis of the proposition implies that the last sum is dominated by constant times
\begin{align*}
\sum_{j=0}^\infty
\left\|
f \,r_j\,  e^{-R(\cdot)}
\right\|_{L^1(du)} =
\sum_{j=0}^\infty
\|
f \,r_j
\|_{L^1(\gamma_\infty)} =
 \|f\|_{L^1(\gamma_\infty)},
\end{align*}
 which proves the assertion.
\end{proof}

We will now
apply  Calder\'on-Zygmund theory to the operator $q_T$, in order to prove its weak type $(1,1)$ with respect to Lebesgue measure. First we verify the  $L^2$ boundedness.

\medskip


\begin{lemma}\label{CalderonZ2}
The operator $ q_T$ is bounded on $L^2(dx)$.
\end{lemma}

\begin{proof}
 Starting with \eqref{p_Tdef} and \eqref{defdefmloc},  we then apply the bounded overlap of the  $\widetilde r_j$ and \eqref{Rconst}.                  
 We get
\begin{align*}
 \int |q_T f(x)|^2dx&=
\int \left|
\sum_{j=0}^\infty e^{-R(x)}\, \widetilde r_j (x)
  m_{0}(T)\left(f\, r_j\, e^{R(\cdot)}\right)(x)
\right|^2\, dx\\
&\lesssim \sum_{j=0}^\infty  \int \left|  e^{-R(x)}\, \widetilde r_j (x)
  m_{0}(T)\left(f\, r_j\, e^{R(\cdot)}\right)(x)
\right|^2\, dx\\
&\lesssim \sum_{j=0}^\infty  e^{-R(x_j)}\, \int \Big|
\widetilde r_j (x)
  m_{0}(T)\left(f \,r_j\, e^{R(\cdot)}\right)(x)
\Big|^2 \, d\gamma_\infty (x)
\\
&\lesssim \sum_{j=0}^\infty  e^{-R(x_j)}\, \int \Big|
  m_{0}(T)\left(f\, r_j \,e^{R(\cdot)}\right)(x)
\Big|^2 \, d\gamma_\infty (x).
\end{align*}

 For $m_{0}(T)$, which is of Laplace transform type, the $L^2$ boundedness with respect to the invariant measure
follows from  \cite[Lemma~3.7]{Carbonaro-Oliver}
(we remark that this boundedness  is also a consequence of some results in \cite{CFMP1} and \cite{CFMP2}, which can be applied here
since  \cite[Lemma 2.2]{MPRS} exhibits a linear change of coordinates in $\R^n$  reducing the setting to the case where  $Q = I$ and $Q_\infty$  is a diagonal matrix).
As a consequence of the above, we have
\begin{align*}
 \int |q_Tf(x)|^2\, dx
 &\lesssim \sum_{j=0}^\infty  e^{-R(x_j)}\, \int \Big|
 f(u) \,r_j(u)\, e^{R(u)}
\Big|^2 \, d\gamma_\infty (u)\\
 &\simeq \sum_{j=0}^\infty   \int \big|
 f(u)\, r_j (u)
\big|^2\, du
\le   \int |
 f (u)
|^2 \,du,
\end{align*}
concluding the proof.
\end{proof}

We need a lemma from \cite{CCS3}.
\begin{lemma}
\label{lemma-integral_dt}
Let $p,\: r\ge 0$ with $p+ r/2 > 1 $. Assume that $\eta(x,u)> 0$  and  $x\neq u$. Then for $\delta>0$
\begin{equation} \label{propclaim}
\int_0^{1}
t^{-p}   \exp\left(-\delta\,\frac{|u- D_t \,x |^2}t\right) |x|^{r} \,  dt\le C\,{|u-x|^{-2p-r+2}}.
\end{equation}
Here the constant $C$ may depend on  $\delta,\: p$ and   $r$, in addition to   $n$, $Q$ and  $B$.
\end{lemma}
\begin{proof}
To see that this follows from \cite[Lemma 8.1]{CCS3}, it is enough to observe that Lemma \ref{lemma1}($i$) leads to $|u-x| \lesssim 1/(1+|x|) $, i.e., $(x,u) \in L_A$ in the terminology of \cite{CCS3}.
   \end{proof}

We shall now find the
 off-diagonal kernel  $\mathcal Q(x,u)$ of the operator $q_T$
 from \eqref{p_Tdef}, defined for integration against Lebesgue measure,  that is,
\begin{equation*}
 q_T f(x)= \int  \mathcal Q(x,u)\,f(u)\,du,   \qquad x \notin \text{supp}\,f.
\end{equation*}

 It will be convenient to introduce another kernel
 \begin{align}\label{defKRtcalligrafico}
\mathcal K_t (x,u)
&:= e^{-R(x)}\,K_t (x,u)\notag\\
&=
\Big(
\frac{\det \, Q_\infty}{\det \, Q_t}
\Big)^{{1}/{2} }\,
\exp \Big[
{-\frac12
\left\langle (
Q_t^{-1}-Q_\infty^{-1}) (u-D_t \,x) \,,\, u-D_t\, x\right\rangle}\Big].\quad\quad
\end{align}

We recall from \eqref{eq:moloc} and Proposition \ref{conv0'} that
 the off-diagonal kernel of
$m_{0}(T)^{{\rm{loc}}}$ is
\begin{equation}\label{M0loc}
  \mathcal M_{0}^{{\rm{loc}}}(x,u)
 = \mathcal M_{0}(x,u)\,\eta(x,u)
   = -\int_0^{1}
\varphi (t)\,
\dot K_t (x,u)
\,dt \;\eta(x,u).
\end{equation}

From this and \eqref{p_Tdef} it follows that
\begin{equation}\label{def:Pcalli}
 \mathcal Q(x,u)=
 e^{-R(x)}\,
\mathcal M_{0}^{{\rm{loc}}} (x,u)
=   -\int_0^1
\varphi (t)\,
\dot {\mathcal K}_t (x,u)
\,dt \:
 \eta (x,u).
\end{equation}

We will need expressions for some derivatives of $\mathcal  K_t$; for similar results about the derivatives of $K_t$ we refer to
\cite[Lemma~4.1]{CCS3}.

 Using \eqref{Dt_again}, one sees that
\begin{align}
\partial_{x_\ell}\, \mathcal K_t (x,u) &=
 \mathcal K_t (x,u)\,  \mathcal P_\ell(t,x,u), \label{derPt}
 \end{align}
 where
\begin{align}
\mathcal  P_\ell(t,x,u)&=  \label{expr1Pell}
{
\left\langle   Q_t^{-1}\,
e^{tB}\, e_\ell
  \,,\, u-D_t\, x\right\rangle}.
  \end{align}
Similarly, or as an immediate consequence of \cite[formula (4.2)]{CCS3},
\begin{align}
\partial_{u_\ell}\, \mathcal  K_t (x,u)
&=
 - \mathcal   K_t (x,u)\,{\left\langle Q_t^{-1}\, e^{tB} \, ( D_{-t}\, u - x)\,, \,e_\ell\right\rangle}.\label{derKtul}
\end{align}

The following three technical lemmata give expressions and estimates for derivatives of
$\dot {\mathcal  K_t} = \partial \mathcal  K_t/\partial t$.
Before stating them, we notice that
\begin{equation}\label{hej}    
\dot {\mathcal  K_t}(x,u)=
e^{-R(x)}\,
\dot {  K_t}(x,u)
= e^{-R(x)}\, K_t(x,u) \, N_t(x,u)=
{\mathcal  K_t}(x,u)\,N_t(x,u),
\end{equation}
with $N_t(x,u)$  from Lemma \ref{derivate-nucleo}.

\begin{lemma}\label{lemma-derivatives}
For $x,\,u\in\R^n$
and $t>0$,  one has
\begin{align*}
&{{(i)}}\qquad
\partial_{x_\ell} \,
\dot {\mathcal  K_t} (x,u)
=\mathcal K_t (x,u)\, \mathcal S_\ell(t,x,u) ; \\
&{{(ii)}}\qquad
\partial_{u_\ell}\, \dot {\mathcal  K_t}(x,u)
= \mathcal K_t (x,u)\, \mathcal R_\ell(t,x,u),
\end{align*}
where the factors $\mathcal S_\ell(t,x,u)$ and $\mathcal R_\ell(t,x,u)$ are given by
\begin{align}
&\mathcal S_\ell(t,x,u)=
 N_t (x,u)\,
\mathcal P_\ell(t,x,u)-
{
\left\langle   Q_t^{-1}\, e^{tB}\, Q\, e^{tB^*}\, Q_t^{-1}\,
e^{tB}\, e_\ell
  \,,\, u-D_t\, x\right\rangle}
\notag\\
&\qquad +
  {
\left\langle   Q_t^{-1}\,B\,e^{tB}\, e_\ell
  \,,\, u-D_t\, x\right\rangle}
  +\left\langle
  Q_t^{-1}\,
e^{tB}\, e_\ell
  \,,\, Q_\infty \,B^*\, e^{-tB^*}\, Q_\infty^{-1}\,  x\right\rangle,
\label{defS1}
\end{align}
and
\begin{align}
 \mathcal R_\ell(t,x,u)
&=
-N_t(x,u)
 \,{\left\langle Q_t^{-1}\, e^{tB} \, ( D_{-t}\, u - x)\,, \,e_\ell\right\rangle}
\notag\\
 +&
 {\left\langle
Q_t^{-1}\, e^{tB}\, Q\, e^{tB^*}\, Q_t^{-1}\, e^{tB} \, ( D_{-t}\, u - x) \,, \,e_\ell\right\rangle}
\notag  \\
-&
{\left\langle
  Q_t^{-1} \, B\,e^{tB} \, ( D_{-t}\, u - x)\, , \,e_\ell\right\rangle}-
{\left\langle
  Q_t^{-1}\, e^{tB} \,   Q_\infty\, B^*\, \,Q_\infty^{-1}  \,D_{-t}\,u\,, \,e_\ell
  \right\rangle} .
\label{def_Rttt}
\end{align}
\end{lemma}

\begin{proof}
   To prove {\it{(i)}}, we  start by observing that
  \begin{align*}
 \partial_{x_\ell}\, \dot {\mathcal  K_t} (x,u)
= & \,{\partial_t} \left( \mathcal K_t (x,u)\, \mathcal P_\ell(t,x,u)\right)\\
 =&\,
\mathcal   K_t (x,u)\, N_t (x,u)\,
\mathcal P_\ell(t,x,u)+
\mathcal   K_t (x,u) \, {\partial_t}\,
 \big(
{
\left\langle   Q_t^{-1}\,
e^{tB}\, e_\ell
  \,,\, u-D_t\, x\right\rangle} \big),
 \end{align*}
 where we used  \eqref{expr1Pell}. Applying  \eqref{8} and  \eqref{derd-t} to the last derivative here,
 one arrives at  \eqref{defS1}, and
  {\it{(i)}} is verified.

  To prove {\it{(ii)}}, we  proceed similarly, using
   \eqref{derKtul} to write
\begin{multline*}
 \partial_{u_\ell}\dot {\mathcal  K_t} (x,u)=
- {\mathcal  K_t} \, N_t(x,u)
 \,{\left\langle Q_t^{-1}\, e^{tB} \, ( D_{-t}\, u - x)\,, \,e_\ell\right\rangle}\\
 -
\mathcal K_t (x,u)\:{\partial_t}\left({\left\langle Q_t^{-1}\, e^{tB} \, ( D_{-t}\, u - x)\,, \,e_\ell\right\rangle}\right).
  \end{multline*}
 As in the case of  {\it{(i)}},  this leads to \eqref{def_Rttt} and  {\it{(ii)}}.
\end{proof}

\medskip

To bound  $\mathcal S_\ell$ and $\mathcal R_\ell$, one concludes from \cite[formula (4.5)]{CCS3} that
(notice the distinction between our $\mathcal P_\ell$ and the $P_j$ used in \cite{CCS3} )
\begin{equation}
 |\mathcal P_\ell (t,x,u)|\lesssim
 {
| u-D_t \,x
|}/{t}, \qquad 0 <t \le 1. \label{est-for-Pj}
\end{equation}

\begin{lemma}\label{stimaStttt}
One has for  $0<t \le1$ and all $x,\,u\in\R^n$
\begin{equation*}
|\mathcal S_\ell(t,x,u)|
 \lesssim
|x|\,\frac{\left|u-D_t \,x\right|^2}{t^2}+
\frac{\left|u-D_t\, x\right|^3}{t^3}
+
\frac{|u-D_t\, x|}t
+\frac{| u-D_t \,x|}{t^2}
+\frac{| x|}{t}.
\end{equation*}
\end{lemma}

\begin{proof}
We first bound the product $  N_t (x,u)\,
\mathcal P_\ell(t,x,u)$ appearing in  \eqref{defS1}.
  Because of \eqref{R1}
 and \eqref{est-for-Pj}, we have for  $0< t\leq1$
\begin{align}
\big| & N_t (x,u)
\mathcal P_\ell(t,x,u)\big|
\: \lesssim\:\left( \frac{1}{t}
+\frac{\left|u-D_t\, x\right|^2}{t^2}
+ |x|\,\frac{|u-D_t \,x|}t\right)\,
 \frac{
| u-D_t \,x
|}{t}\notag\\
& \lesssim
\frac{| u-D_t\, x|}{t^2}
+
\frac{\left|u-D_t\, x\right|^3}{t^3}+
|x|\,\frac{\left|u-D_t\, x\right|^2}{t^2}
\label{esr-NtPl}.
\end{align}
Estimating also the other terms in \eqref{defS1}, one arrives at the lemma.
                     \end{proof}

\begin{lemma}\label{stimaRtttt}
For $t\in (0,1]$ and all $x,\,u\in\R^n$
\begin{equation*}
|\mathcal R_\ell(t,x,u)|
 \lesssim
   |x|\,\frac{|u-D_t\, x|^2}{t^2}+\frac{| u-D_t \,x|^3}{t^3}+
\frac{| u-D_t\, x|}{t^2}
+
+\frac{|x|}t.
\end{equation*}
\end{lemma}
\begin{proof}
 For $t\in (0,1]$ we have  by \eqref{def_Rttt} and  \eqref{R1}
\begin{align*}
\big| \mathcal R_\ell(t,x,u)\big|
&\lesssim
\big|N_t(x,u)
\big|
 \,\frac{| u-D_t\, x|}{t}
+\frac{| u-D_t\, x|}{t^2}+ \,\frac{| u-D_t\, x|}{t}+\frac{|u|}t\\
&\lesssim
\left( \frac{1}{t}
+\frac{\left|u-D_t\, x\right|^2}{t^2}
+ |x|\,\frac{|u-D_t \,x|}t   \right)
\frac{| u-D_t\, x|}{t}
+\frac{| u-D_t \,x|}{t^2}+\frac{|x|}t.
\end{align*}
Here we estimated $|u|/t$ by $| u-D_t\, x|/t^2 + |x|/t$. The lemma follows.
\end{proof}

\begin{proposition}
\label{lemma-Calderon-prelim}
For  all $(x,u)$ such that $\eta(x,u)\neq 0$  and  $x\neq u$,
  one has
\begin{align*}
\int_0^{1}
\big|
\dot {\mathcal  K_t}(x,u)\big|
\,dt
& \lesssim  |u-x|^{-n}.
\end{align*}
                 \end{proposition}
\begin{proof}
From \eqref{hej}    
\eqref{litet} and \eqref{R1} we obtain
\begin{align*}
\int_0^1
& \big|
\dot {\mathcal  K_t} (x,u)\big|
\,dt \\
&\lesssim
 \int_0^1 t^{-\frac n2} \exp\left(- c\,\frac{|D_t\,x -u|^2}t\right)\,
\left(\frac{1}{t}
+\frac{\left|u-D_t\, x\right|^2}{t^2}
+ |x|\,\frac{|u-D_t\, x|}t\right) dt\\
&\lesssim
 \int_0^1 t^{-\frac n2} \exp\left(- c\, \frac{|D_t\,x -u|^2}t\right)\,
\left(\frac{1}{t}
+\frac{|x|}{\sqrt t}
\right)\, dt.
\end{align*}
Because of
Lemma \ref{lemma-integral_dt},
 the last  integral  is controlled
 by $  |u-x|^{-n}$, and the proposition
is proved.
\end{proof}

We are now ready to prove standard   Calder\'on-Zygmund bounds
for
 the off-diagonal kernel of $q_T$  with respect to Lebesgue measure.

\begin{proposition}\label{lemma-Calderon}
For  all $(x,u)$ such that $\eta(x,u)\neq 0$  and  $x\neq u$,
the following estimates hold:
\begin{align}
	\big| \mathcal Q(x,u)\big| 
	&\lesssim   |u-x|^{-n}; \label{stima-Mphiloc}\\
 \big|\nabla_x\, \mathcal Q(x,u)\big|&\lesssim   {|u-x|^{-n-1}};\label{stima-Mphiloc-Dx}\\
\big|\nabla_u \,
\mathcal
Q(x,u)
\big|
& \lesssim { }{|u-x|^{-n-1}}.\label{stima-Mphiloc-Du}
\end{align}
\end{proposition}
\begin{proof}
In the light of \eqref{def:Pcalli}                  
one has
\begin{align*}
| \mathcal Q(x,u)|&=
\left|\int_0^{1}
\varphi (t)\,
\dot {\mathcal K}_t (x,u)
\,dt \right|\;\eta(x,u)\le
\int_0^{1}
\big|
\dot {\mathcal K}_t (x,u)\big|
\,dt
 \lesssim  |u-x|^{-n},
\end{align*}
where we used both \eqref{eta_ineq} and Proposition  \ref{lemma-Calderon-prelim}.
Thus \eqref{stima-Mphiloc} is verified.


In order to prove
\eqref{stima-Mphiloc-Dx},
we first observe that  
\begin{equation}  \label{product}
\big|
\partial_{x_\ell}\,\left( \mathcal Q
(x,u)\right)
\big|
\lesssim
 \int_0^{1}
\big|
\partial_{x_\ell}\,
\dot {\mathcal K}_t (x,u)
\big|
\,dt
\,\eta(x,u)+\int_0^{1}
\big|
\dot {\mathcal K}_t (x,u)\big|
\,dt
  \,\big|\partial_{x_\ell}\,\eta(x,u)\big|.
 \end{equation}
 The last term here satisfies the desired estimate because of Proposition \ref{lemma-Calderon-prelim} and
 Lemma \ref{lemma1}($ii$).

 Using  Lemma \ref{lemma-derivatives}{{$(i)$}} and then
 \eqref{litet} and Lemma \ref{stimaStttt}, we can estimate the first  term in \eqref{product}
 by
\begin{align*}
&\int_0^{1}
\big|
{\mathcal K}_t (x,u)\, \mathcal S_\ell (t,x,u)\big|
\,dt
\lesssim
\int_0^{1}
 {t^{-\frac n2}} \,\exp\left(-c\,\frac{|u-D_t\, x |^2}t\right)\\
&\quad\times \left(
|x|\,\frac{\left|u-D_t x\right|^2}{t^2}+
\frac{\left|u-D_t x\right|^3}{t^3}
+
\frac{|u-D_t x|}t
+\frac{| u-D_t x|}{t^2}
+\frac{| x|}{t} \right)
\,dt\\
&\qquad\qquad \lesssim \int_0^{1}
 { }{t^{-\frac n2}}\, \exp\left(-c\,\frac{|u-D_t\, x |^2}t\right)
 \left( \frac{| x|}{t} + \frac{1}{t\sqrt t}
 \right)
\,dt.
\end{align*}
Proposition
\eqref{lemma-integral_dt} says that
 the last expression is controlled
 by $ |u-x|^{-n-1}$, so that \eqref{stima-Mphiloc-Dx} is proved.

The verification of  \eqref{stima-Mphiloc-Du} is analogous, 
with  $\mathcal R_\ell (t,x,u)$ instead of
 $\mathcal S_\ell (t,x,u)$. It is enough to observe that  Lemma \ref{stimaRtttt} implies that
  $|\mathcal R_\ell (t,x,u)|$   is controlled by the right-hand side in the statement of Lemma \ref{stimaStttt}.
\end{proof}

By means of Lemma \ref{CalderonZ2}, Proposition \ref{lemma-Calderon} and Proposition \ref{prop81}
we finally arrive at the goal of this section.

\begin{proposition}\label{propo-locale}
The operator   $m_{0}(T)^{{\rm{loc}}}$
is
 of weak type $(1,1)$ with respect to the invariant measure $d\gamma_\infty$.
\end{proposition}

      \section{An auxiliary bound for $0<t \leq 1$}\label{number-of-zeros}


In this section, we verify a bound on
 the number of zeros of the $t$ derivative of  $K_t$ in the interval $(0,1]$, which will
 be used in the next section to control  the kernel $\mathcal M_{0}^{{\rm{glob}}}$.

\begin{proposition}\label{labour day}
For $(x,u)\in \mathbb R^n\times\mathbb R^n$, the
 number of zeros in $I=(0,1]$ of the function  $t \mapsto \dot K_t(x,u)$ is bounded
by a positive integer depending only on   $n$ and $B$.
\end{proposition}

\begin{proof}
Instead of $ \dot K_t(x,u)$ we consider $\mathcal N_t(x,u) = 2(\det Q_t)^2\, N_t (x,u)$,
since the three kernels $\dot K_t(x,u)$, $ N_t(x,u)$ and $\mathcal N_t(x,u)$ have exactly the same zeros in $I$.
 From \eqref{R} we have
\begin{align}
\mathcal N_t(x,u)=&-{(\det Q_t)}\,{\tr\big((\det Q_t)Q_t^{-1} \, e^{tB}\, Q\, e^{tB^*}\big)} \label{N'}
\\\notag
&+ \left\langle Q\, e^{tB^*}\,(\det Q_t)Q_t^{-1}\,(u-D_t\, x)\,,\,e^{tB^*}\,(\det Q_t)Q_t^{-1}\,
(u-D_t\, x)\right\rangle
\\ \notag
&
- 2(\det Q_t)
\left\langle Q_\infty \,B^*\, Q_\infty^{-1}\, D_t\, x\,,\,
\left((\det Q_t)Q_t^{-1}-(\det Q_t)Q_\infty^{-1}\right)(u-D_t\, x)\right\rangle;
\end{align}
notice that here we have placed a factor $\det Q_t$ at each occurrence of $Q_t^{-1}$.


We denote by  $\nu_j, \; j = 1,\dots, J$ the eigenvalues of
 $B$, and observe that those which are nonreal  come in conjugate pairs, and that
  $\Re \nu_j < 0$ for all $j$.

\begin{claim}\label{claim2'}
The function $t \mapsto \mathcal N_t(x,u)$ is a finite linear
combination, with coefficients depending on $(x,u)$, of terms which are given by
a product of type $\prod_{j=1}^J e^{m_j\nu_j t }$ multiplied by a polynomial in $t$
with complex coefficients. Here
 $m_j \in \mathbb Z$. Further,  the  quantities $|m_j|$   and the
  degrees of the polynomials are all bounded by  a constant depending only on
  $n$.   This bound also applies to the number of terms.
  \end{claim}

\begin{proof}
  Inspection shows that the last two terms in \eqref{N'} are sums of scalar products of vectors given by multiplying $x$ or $u$ from the left by various combinations
of the matrices $e^{tB}$, $e^{tB^*}$, $D_{ t}$, $Q_t$ and $(\det Q_t) Q_t^{-1}$,
the constant matrices $B^*$, $Q$, $Q_\infty$ and $Q_\infty^{-1}$, and the scalar factor  $\det Q_t$.
The first term in \eqref{N'} is instead the trace of the product of some of these matrices, multiplied by $\det Q_t$.
Let us examine precisely how the matrices listed here depend on $t$.

We pass from $\mathbb R^n$ to $\mathbb C^n$ and make a Jordan decomposition of $B$ via a change of coordinates in $\mathbb C^n$. Each Jordan block is of the form $\nu_j(I + R)$, where $R$ is a supertriangular and thus nilpotent matrix and $I$ is the identity matrix, of some dimension.
Then $\exp(t\nu_j(I + R)) = e^{\nu_j t } P(t)$, where $P(t)$ is a matrix with polynomial
entries in $t$. To arrive at $\exp (tB)$, we put these blocks together and then change coordinates back.
The result will be that
 in the coordinates we had before, each entry of the matrix $\exp (tB)$ is
a sum over $j$ of terms  of type $e^{\nu_j  t } p(t)$, where  $p(t)$ is a complex polynomial that may depend on $j$ and on the entry considered.
  The same will be true
  for the entries of   its adjoint $\exp (tB^*)$.
 From \eqref{defDt} we then see that
  $D_{t}$ is of the same form but with  $e^{-\nu_j  t } $
 instead of $e^{\nu_j  t } $.
        Considering the integral in \eqref{defQt}, we see that the matrix $Q_t$ has similar entries, now with
terms  $e^{(\nu_j +\nu_{j'}) t } p(t)$.  Since the entries of the matrix $(\det Q_t) Q_t^{-1}$
are given by minors of $Q_t$, they will be  a sum of terms which are like those described in Claim \ref{claim2'}. Finally, the scalar $\det Q_t$ also has the same structure.

Claim \ref{claim2'}  now follows, since the
bound on the $|m_j|$ and  the degrees of the polynomials is easily verified.      \end{proof}

We observe that Claim \ref{claim2'} implies that  $\mathcal N_t(x,u)$ can be extended to an entire function
in $t$, and so the number of zeros in $(0,1]$ is finite.

This claim means that  $ \mathcal N_t(x,u)$ is a sum of terms given by a function of $(x,u)$ times an expression
 of type
  \begin{equation}\label{term}
   \exp\left(\sum_j m_j \nu_j t\right) P(t)  = \exp\left( (\lambda + i\mu)t\right) P(t),
  \end{equation}
  where we write  $ \sum_j m_j \nu_j  = \lambda + i\mu$ and $P(t)$ is a complex
   polynomial.

  We will now find a linear differential operator in $t$,  independent of $x$ and $u$, that annihilates  all these expressions and thus also
  $ \mathcal N_t(x,u)$, for all $(x,u)$.   For this we denote  $D = d/dt$.

  If $\mu = 0$         
  the expression in \eqref{term} is annihilated by
   \begin{equation*}
   \left(D-\lambda\right)^{1+\text{deg}\,P}.
  \end{equation*}
  If  $\mu \ne 0$, the same expression is annihilated by the operator
    \begin{equation*}
   (D-\lambda -i\mu)^{1+\text{deg}\,P}.
  \end{equation*}
  Since   $\mathcal N_t(x,u)$  coincides with its real part, there is also a term
   \begin{equation}\label{termconj}
   \exp\left((\lambda -i\mu) t\right)\bar P(t),
  \end{equation}
  again multiplied by a function of $(x,u)$, in the sum forming   $\mathcal N_t(x,u)$. This term is annihilated by
   \begin{equation*}
   (D-\lambda +i\mu)^{1+\text{deg}\,P}.
  \end{equation*}
  Clearly both terms \eqref{term} and \eqref{termconj} are annihilated by the product of the two operators, which is
   \begin{equation*}
  \left( (D-\lambda)^2 +\mu^2)\right)^{1+\text{deg}\,P}.
  \end{equation*}

  Consider now all the terms in the sum giving   $\mathcal N_t(x,u)$.
  It folllows that  $\mathcal N_t(x,u)$ is annihilated by a differential operator
 \begin{align*}
\mathcal P(D) = \prod_{i=1}^{K} P_i(D),
\end{align*}
where each $P_i(D)$ is of either of
 the following two types: a first-order operator
\begin{align*}
T_\lambda= D -\lambda
\end{align*}
or a second-order  operator of the form
\begin{align*}
S_{\lambda,\mu} = (D -\lambda)^2 +\mu^2.
\end{align*}
Here $\lambda\in \mathbb R$ and $\mu \ne 0$.
Clearly all these operators commute, and  $\mathcal P(D)$ is a polynomial in $D$ with coefficients depending only on
$n$ and $B$, and with leading coefficient 1. The number of factors in  $\mathcal P(D)$  has a bound also depending only on $n$ and $B$.

Without restriction, we may assume that $\mu >0$ in each operator  $S_{\lambda,\mu}$, and also that the equation
 $\mathcal P(D)\mathcal N_t(x,u) = 0$ does not allow suppression of any of the factors
  $P_i(D)$  in the  product defining  $\mathcal P(D)$.

Proposition \ref{labour day} is thus reduced
to showing that the number of zeros of  a real-valued  solution of the equation
$\mathcal P(D) \phi = 0$
in $I = (0,1]$ is bounded by a constant depending only on the polynomial
 $\mathcal P$.

Our next claim deals
with one operator $T_\lambda$ or  $S_{\lambda,\mu}$.

\begin{claim}\label{claim 5'}
Let $\lambda \in\mathbb R$ and $\mu>0$, and let  $J\subset \mathbb R$ be a closed interval  of
length less than $1/{\mu}$. Assume that $\phi \in C^2(J)$ is a real-valued function.
If $S_{\lambda,\mu}\,\phi$ does not vanish in the interior  $J^\circ$ of $J$,
then $\phi$ has at most two zeros in $J$.
Further, if $S_{\lambda,\mu}\phi$ has at most $k$
zeros in $J$, then $\phi$ has at most $2k+2$
zeros in the same interval.
The same statements  hold with $S_{\lambda,\mu}$ replaced by $T_\lambda$.
\end{claim}

\begin{proof}
To prove the first assertion about $S_{\lambda,\mu}\,,$
we may take $\lambda =0$ since
\begin{equation*}
  S_{0,\mu}\, \phi (t) = e^{-\lambda t}\, S_{\lambda,\mu} \big( e^{\lambda t} \, \phi(t) ),
\end{equation*}
and we will write $S_{\mu}$ for
$S_{0,\mu}$. The same trick applies to $T_\lambda$.


Since by hypothesis
 $S_\mu\,\phi\neq0$ in $J^\circ$, we may as well take  $S_\mu\,\phi > 0$ there.
We assume by contradiction that  $t_1<t_2<t_3$ are three  zeros of
$\phi$ in $J$.
Then
$\phi''(t_2)=S_\mu\,\phi(t_2)>0$.
We can then assume that $\phi'(t_2)\geq0$, since otherwise
we consider instead the function $\phi(-t)$
in the interval $-J$.
For $t>t_2$ sufficiently close to $t_2$ we  have
$$
\phi(t) =\phi'(t_2)(t-t_2)+\frac12 \,\phi''(t_2) (t-t_2)^2 + o\left((t-t_2)^2\right)
>0.
$$
 Since $\phi(t_3)=0$, the maximal value $M$ of $\phi$ in the interval $[t_2,t_3]$ must be
 assumed at some point $t_M \in (t_2,t_3)$. Clearly $M>0$ and  $\phi'(t_M)=0$.
An integration by parts yields
\begin{align*}
M& = \int_{t_2}^{t_M} \phi'(t)\,dt
= (t-t_2) \phi'(t)|_{t_2}^{t_M}
- \int_{t_2}^{t_M} (t-t_2) \phi''(t)\,dt\\
&= - \int_{t_2}^{t_M} (t-t_2) \phi''(t)\,dt.
\end{align*}

Since here $-\phi''(t) = \mu^2\phi(t) - S_\mu\,\phi(t) < \mu^2\phi(t) \le \mu^2 M$
we conclude that
\begin{align*}
M \leq \mu^2 M \int_{t_2}^{t_M} (t-t_2)dt
=\mu^2 M\, \frac{(t_M-t_2)^2}2
\leq \frac12\,\mu^2 M |J|^2.
\end{align*}
This leads to the contradiction $|J|\geq \sqrt2/\mu$, which proves  the first assertion of the claim.
 The second assertion follows from the first, applied in each of the intervals obtained by deleting from $J$ the zeros of $\phi$.

 For $T_\lambda$ it is enough to apply Rolle's theorem to  $T_0 = D$.
 \end{proof}

\noindent {\it{Conclusion of the proof of Proposition \ref{labour day}}}.
We know that
 $\mathcal P(D)\mathcal N_t(x,u) = 0$,  where
\begin{align*}
\mathcal P(D) = \prod_{i=1}^{K} P_i(D),
\end{align*}
and each $P_i(D)$ is
  $T_{\lambda}$ or $S_{\lambda,\mu}$ for some
with $\lambda\in\R$,  and $\mu>0$ in the case of $S_{\lambda,\mu}$.

Choose a natural number $\kappa$  larger than all the values of $\mu$ appearing here.
Then split $[0,1]$ into $\kappa$ closed subintervals of length $1/\kappa$, and let $J$ be one of these
subintervals. Observe that  Claim  \ref{claim 5'} applies to each $P_i(D)$ in $J$, since  $1/\kappa < 1/\mu$.


Set for $m\in\{2,3,\ldots,K\}$
\begin{align*}
\mathcal N^{(m)}_t(x,u)=\prod_{i=m}^K P_i(D)\; \mathcal N_t(x,u),
\end{align*}
and $\mathcal N^{(K+1)}_t(x,u) =\mathcal N_t(x,u)$.

We will prove by induction that the function $t\mapsto\mathcal N^{(m)}_t(x,u)$ has at most $2^{m}-2$
zeros in $J$, for $m\in\{2,3,\ldots,K+1\}$. Here we fix $(x,u)$. Proposition \ref{labour day} will then follow from the case $m=K+1$.

 Starting with $m=2$, we have
$P_1(D)\, \mathcal N^{(2)}_t(x,u) = 0$, and $\mathcal N^{(2)}_t(x,u)$ is not identically 0  for all $t$.
By means of a conjugation with the factor $e^{\lambda t}$ as in the proof of Claim~\ref{claim 5'}, we can assume that  $P_1(D)$ is either $T_0 = D$ or $S_{0,\mu}$.
If  $P_1(D) = D$, then $\mathcal N^{(2)}_t(x,u)$ is a nonzero constant;  if $P_1(D) = S_{0,\mu}$ we assume that $t=t_0 \in J$ is a zero
of $ \mathcal N^{(2)}_t(x,u)$. Then $ \mathcal N^{(2)}_t(x,u)$ is proportional  to $\sin\left((t-t_0){\mu}\right)$ and can have no other zero
 in  $J$, because $|J| < 1/ \mu$. The first induction step is verified.

Assume the induction step holds for $m$. Then
$P_{m}(D)\,\mathcal N^{(m+1)}_t(x,u) = \mathcal N^{(m)}_t(x,u)$ has at most  $2^{m}-2$ zeros in $J$, and Claim  \ref{claim 5'} implies that the number of zeros of $\mathcal N^{(m+1)}_t(x,u) $ in $J$ is at most
$2(2^{m}-2)+2 = 2^{m+1}-2$. The induction is complete, and so is the proof of Proposition \ref{labour day}.
\end{proof}

\section{Estimates in the global region for small $t$ }\label{est-tsmall-glob-section}

In this section we estimate the operator  $m_{0}^{{\rm{glob}}}(T)$ with kernel
\begin{equation*}
 -\int_0^1
\varphi (t)\dot K_t (x,u) \big(1-\eta(x,u)\big)
\,dt.
\end{equation*}

We shall need
the following theorem.
In order not to burden the exposition, we postpone its proof to the appendix.

\begin{theorem}\label{stima per M con t piccolo}
The maximal operator defined by
\begin{align*}
S_{0}^{{\rm{glob}}}
 f(x) = \int \sup_{0<t\leq 1}K_t(x,u)\,
\big(1-\eta(x,u)\big) \,|f(u)|\, d\gamma_\infty(u)
\end{align*}
is
 of weak type $(1,1)$ with respect to the invariant measure $d\gamma_\infty$.

\end{theorem}

This is a sharpened version of  the weak type $(1,1)$ estimate for the corresponding part of the maximal
operator treated in \cite{CCS2}, since the supremum in $t$ is now placed inside the integral.
As a consequence,
we can prove the following result, which will complete the proof of Theorem \ref{weaktype1}.

\begin{proposition}\label{propo-mixed-glob-enhanced}
The operator   $m_{0}^{\rm{glob}}(T)  $
is of weak type $(1,1)$ with respect to the invariant measure $d\gamma_\infty$.
\end{proposition}

\begin{proof}
[Proof of Proposition \ref{propo-mixed-glob-enhanced}]
Let $N(x,u)$ be the number
of zeros in $(0,1)$ of the function $t\mapsto \dot K_t(x,u)$.
 Proposition \ref{labour day} says that $N(x,u) \le  \bar N$  for some  constant $\bar N \ge 1$
that is  independent of   $(x,u) \in \R^n \times \R^n$ (and dependent only of $n$ and $B$).
We denote these zeros by $t_1(x,u)< \cdots< t_{N(x,u)}(x,u)$,
and set $t_0(x,u) =0$, \hskip4pt
$t_{N(x,u)+1}(x,u) =1$.
Since $K_t(x,u)$ vanishes at $t=0$,
it follows from the fundamental theorem of calculus that
\begin{multline*}
\int_0^1 \left|\dot K_t(x,u)\right| dt
= \sum_{i=0}^{N(x,u)} \left|\int_{t_i(x,u)}^{t_{i+1}(x,u)} \dot K_t(x,u) dt\right|
\\
= \sum_{i=0}^{N(x,u)} \left|K_{t_{i+1}(x,u)}(x,u) - K_{t_i(x,u)}(x,u)\right|
\\
\leq 2 \sum_{i=0}^{N(x,u)+1}K_{t_i(x,u)}(x,u)
\;\lesssim \;\bar N \sup_{0<t\leq 1}K_{t}(x,u).
\end{multline*}
This inequality implies
\begin{align*}
 \left|m_{0}^{{\rm{glob}}}(T) f(x)\right|
& \leq
\int \int_0^1
\left|\dot K_t (x,u)\right|\, dt \,\left(1-\eta(x,u)\right)
 |f(u)|\, d\gamma_\infty(u)
\\
&\lesssim \bar N\,
\int \sup_{0<t\leq 1} K_{t}(x,u)\,
\left(1-\eta(x,u)\right)
 |f(u)|\, d\gamma_\infty(u),
\end{align*}
and Theorem \ref{stima per M con t piccolo} yields
\begin{align*}
\gamma_\infty\left\{x: \left|m_{0}^{{\rm{glob}}}(T) f(x)\right| > \alpha\right\}
\lesssim \frac1\alpha
\int |f(u)|\, d\gamma_\infty (u).
\end{align*}
 \end{proof}

\section{Appendix: Proof of Theorem \ref{stima per M con t piccolo}}\label{Appendix}
In the proof of this theorem, we take $f\ge 0$  normalized in $L^1(\gamma_\infty)$.
All the simplifications introduced in Subsections~\ref{Simplifications} and \ref{simply}  will apply.
In particular, we let $\alpha$ be large, and we need only consider points  $x$ in $\mathcal C_\alpha$; thus
$|x| \gtrsim 1$.
We will write  $x$ and $u \ne 0$
as
$x=D_s \,\tilde x $
and
$u=D_\sigma \,\tilde u $, respectively,
where $\tilde x,\:\tilde u \in \ellipses_{\beta}$ with $\beta = (\log\alpha)/2$
and $s \ge 0,\;\sigma\in\R$.

\begin{lemma}\label{sup small t}
Let  $(x,u)$
be such that  $\eta(x,u)<1$,
 and  let $x\in\mathcal C_\alpha$.
Then
\begin{align*}
\sup_{0<t\leq 1}K_t(x,u)
&\lesssim
e^{R(x)} \min \left(
|\tilde u-\tilde x|^{-n},\, |x|^n
\right).
\end{align*}
\end{lemma}

\begin{proof}
For the first bound, we use \cite[Lemma 4.3(i)]{CCS2}
to get $|D_t\, x - u| \gtrsim |\tilde x-\tilde u|,$
which by \eqref{litet}  yields
\begin{align*}
\sup_{0<t\leq 1}K_t(x,u)
\lesssim
e^{R(x)} \sup_{0<t\leq 1} t^{- n/2}\, \exp\left(- c\,\frac{|\tilde x-\tilde u|^2}t\right)
\lesssim
e^{R(x)} \,|\tilde x-\tilde u|^{-n}.
\end{align*}
To get the second bound, we deduce from Lemma \ref{lemma2}
and the condition  $\eta(x,u)<1$ that
$ |x-u|\gtrsim \frac1{1+|x|} \simeq \frac1{|x|}$ . Then by Lemma~\ref{differ} we have
\begin{align*}
|x|^{-1}
\lesssim |x- u|
 \le
|x-D_t\, x| +|D_t\, x -u|
\lesssim
t |x|+|D_t\, x -u|.
\end{align*}
Thus  $|x|^{-1} \lesssim t |x|$ or $|x|^{-1} \lesssim |D_t \,x -u|$. In the first case,
$ t^{- n/2} \lesssim |x|^n $, and the desired estimate is immediate from \eqref{litet}.
In the second case,
\begin{align*}
K_t(x,u)
 \lesssim
e^{R(x)} \, t^{-\frac n2}\, \exp\left(- \frac c{t|x|^2}\right)
 \lesssim
e^{R(x)}\, |x|^{n}.
\end{align*}
The lemma is proved.
\end{proof}


To continue the proof of Theorem     \ref{stima per M con t piccolo},
  we conclude from  Lemma \ref{sup small t} that for $x\in\mathcal C_\alpha$
\begin{align*}
S_{0}^{{\rm{glob}}} f(x)
\lesssim e^{R(x)} \int
\min \left(|\tilde u-\tilde x|^{-n},\, |x|^n\right) f(u)\, d\gamma_\infty(u)
=A(x)+B(x),
\end{align*}
where
\begin{align} \label{defA}
A(x) &= |x|^n\, e^{R(x)}
\int_{\{u:\,|x| \leq |\tilde u-\tilde x|^{-1}\}}
f(u)\, d\gamma_\infty(u)
\end{align}
and
\begin{align*}
B(x) &= e^{R(x)} \int_{\{u:\,|x| > |\tilde u-\tilde x|^{-1}\}}
|\tilde u-\tilde x|^{-n}\, f(u) \,d\gamma_\infty(u).
\end{align*}
We will
show that
\begin{align}\label{stime per A}
\gamma_\infty\left\{
x\in\mathcal C_\alpha: A(x) >
{\alpha}
\right\}
\lesssim  \alpha^{-1}
\end{align}
and
\begin{align}\label{stime per B}
\gamma_\infty\left\{x\in\mathcal C_\alpha: B(x) >
{\alpha}\right\}
\lesssim  \alpha^{-1}.
\end{align}

Starting with \eqref{stime per A}, we first observe that $A(\tilde x) < {\alpha}$ for
 $\tilde x \in E_\beta$ with $\beta = (\log \alpha)/2$, because
\begin{equation*}
  A(\tilde x) \le |\tilde x|^n\, e^{R(\tilde x)}
\int_{\R^n}
f(u)\, d\gamma_\infty(u) \lesssim (\log \alpha)^n \,\sqrt \alpha < \alpha \,,
\end{equation*}
and $\alpha$ is large.
Further, $x = D_s\,\tilde x \in \mathcal C_\alpha$ implies $0 < s \lesssim 1$ in view  of
\cite[formula (4.3)]{CCS2}.
Let
\begin{equation*}
E_\beta^0 = \{ \tilde x \in E_\beta:   A(D_s\,\tilde x) >\alpha      \;  \;\mathrm{for \hskip5pt some}\; s>0  \;\mathrm{with}
\; \;  D_s\,\tilde x \in \mathcal C_\alpha           \},
\end{equation*}
and define for $\tilde x \in E_\beta^0$
\begin{equation*}
s_0(\tilde x) = \inf \{s: D_s\,\tilde x \in \mathcal C_\alpha \; \;\mathrm{and}
\;\; A(D_s\,\tilde x) >\alpha\}.
\end{equation*}
Then  $0 < s_0(\tilde x) \lesssim 1$ and $A(D_{s_0(\tilde x)} \,\tilde x) = \alpha$.
Moreover, if $A(D_s\,\tilde x) >\alpha$ for some
$D_s\,\tilde x \in \mathcal C_\alpha$, then $\tilde x \in E_\beta^0$ and $s > s_0(\tilde x)$.
In the set $\mathcal C_\alpha$, the expression \eqref{def:leb-meas-pulita} for the Lebesgue measure yields
$dx \simeq \sqrt {\log\alpha} \, dS_\beta\,ds$, and so
\begin{equation*}
\gamma_\infty\left\{
x\in\mathcal C_\alpha: A(x) > {\alpha} \right\}
\lesssim \sqrt {\log\alpha}\,\int_{E_\beta^0} \, \int_{s_0(\tilde x)}^C
e^{-R(D_s\,\tilde x)}\,ds\,dS_\beta(\tilde x).
\end{equation*}
 We now write ${R(D_s\,\tilde x)} = {R(D_{s_0(\tilde x)}\,\tilde x)}\, +\,
{R(D_s\,\tilde x)} \,- \,{R(D_{s_0(\tilde x)}\,\tilde x)}$
and apply  the Mean Value Theorem to the difference here, observing that
               $\partial_s R(D_s \,\tilde x)
\simeq |D_s\, \tilde x|^2
\simeq \log \alpha$
               because of  \cite[formula (4.3)]{CCS2}.
This leads to
\begin{align} \label{levelset}
\gamma_\infty\left\{
x\in\mathcal C_\alpha: \, A(x) > {\alpha} \right\}
 &\lesssim \,\sqrt {\log\alpha} \int_{E_\beta^0} \,e^{-R(D_{s_0(\tilde x)}\,\tilde x)}\,
\int_{s_0(\tilde x)}^\infty
e^{-c (s-s_0(\tilde x))\log\alpha}\,ds\,dS_\beta(\tilde x)    \notag     \\
 &\lesssim \, \frac 1 {\sqrt {\log\alpha}}\, \int_{E_\beta^0} \,
e^{-R(D_{s_0(\tilde x)}\,\tilde x)}\,dS_\beta(\tilde x).
\end{align}

To deal with  the last expression here, we insert the factor $\alpha^{-1}\,A(D_{s_0(\tilde x)}\tilde x) = 1$ in the integral, using the definition \eqref{defA} of $A(.)$. The two exponentials will then cancel.
 We also use the fact that
$\frac12\sqrt {\log\alpha} \le |D_{s_0(\tilde x)}\tilde x| \le 2\sqrt {\log\alpha}$, both to rewrite the first factor in this definition and to extend the domain of integration in $u$ to  $\{u:\,|\tilde u-\tilde x|\le 2 {(\log\alpha)^{-1/2}}\}$. The result is that the last quantity in \eqref{levelset} is at most constant times
\begin{align*}
 \frac1\alpha \,{(\log\alpha)^\frac{n-1}2}
\int_{{ \ellipses_\beta^0}}
\int_{\{u:\,|\tilde u-\tilde x|\le 2 {(\log\alpha)^{-1/2}}\}}
f(u)\, d\gamma_\infty(u)\,
dS_\beta(\tilde x)
& = \\
\frac1\alpha\,{(\log\alpha)^\frac{n-1}2}
\int
 f(u)
\int_{\{\tilde x:\,|\tilde u-\tilde x|\le 2
{(\log\alpha)^{-1/2}}
\}}\,
dS_\beta(\tilde x)\,
d\gamma_\infty(u)
& \lesssim  \,
\frac1\alpha\,
\int f(u)\, d\gamma_\infty(u)
=
\frac1\alpha.
\end{align*}

This proves \eqref{stime per A}, and we move to \eqref{stime per B}. Here
we similarly have $B(\tilde x) < \alpha$ for  $\tilde x \in E_\beta$, and we can define
$ E_\beta^0$  and $s_0(\tilde x)$ as above,  replacing   $A(.)$ by $B(.)$. The rest of the argument is
only slightly different from that for \eqref{stime per A}; we now have
\begin{multline*}
\gamma_\infty \{x\in\mathcal C_\alpha:\,  B(x) >
\alpha \}
\lesssim
\frac1{\sqrt{\log\alpha}}\,
\int_{ \ellipses_\beta^0}
\exp{\left(- R(D_{s_0(\tilde x)}\,\tilde x)\right)}
\,dS_\beta(\tilde x)
\\
 \quad  \qquad \lesssim
\frac1{\alpha}\,
\frac1{\sqrt{\log\alpha}}\,
\int_{ \ellipses^0_\beta}
\int_{\{u:\,|\tilde u-\tilde x| >
(\log\alpha)^{-1/2}/2\}}
|\tilde u-\tilde x|^{-n} f(u)\, d\gamma_\infty(u)\,
dS_\beta(\tilde x)
\\
=
\frac1\alpha\,
\frac1{\sqrt{\log\alpha}}\,
\int f(u)
\int_{\{\tilde x:\,|\tilde u-\tilde x| >
(\log\alpha)^{-1/2}/2 \}}\,|\tilde u-\tilde x|^{-n}\,
dS_\beta(\tilde x)\,
d\gamma_\infty(u)
\lesssim
\frac1\alpha.
\end{multline*}
This is \eqref{stime per B}, and Theorem \ref{stima per M con t piccolo} is proved.

\medskip

In order to  prove Proposition \ref{propo-mixed-glob-enhanced},
Theorem \ref{stima per M con t piccolo}
is enough, as we saw in the preceding section.
However, we take the opportunity to give
the following related result, which strengthens Theorem \ref{stima per M con t piccolo} and also
Theorem 1.1 in
 \cite{CCS2}                         
and  may be of independent interest.
\begin{theorem}\label{th-full-maximal}
The operator $ S^{{\rm{glob}}}$ defined by
\begin{align}\label{da stimare full}
 S^{{\rm{glob}}} f(x) = \int \sup_{0<t<\infty}K_t(x,u)
\,(1-\eta(x,u))
 |f(u)|\, d\gamma_\infty(u), \qquad f\in L^1(\R^n),
\end{align}
is of  weak type $(1,1)$ for the measure $d\gamma_\infty$.
\end{theorem}

This result is a consequence of
 Theorem \ref{stima per M con t piccolo}
and the following proposition.

\begin{proposition}\label{stima per M con t piccolovv}
The operator $S_{\infty}$, defined by
\begin{align*}
S_{\infty}
 f(x) = \int \sup_{t\geq 1}K_t(x,u)\,
| f(u)| \, d\gamma_\infty(u),
\end{align*}
satisfies the  inequality
\begin{equation}
\gamma_\infty\{x: S_{\infty} f(x)>\alpha\}
\lesssim \frac1{\alpha\,\sqrt{\log \alpha}}   \label{3}
\end{equation}
for all normalized functions $f$ in $L^1(\gamma_\infty)$
and all $\alpha >2$.
\end{proposition}

\begin{proof}
Let $t\ge 1$.
The simplifications in Subsections \ref{Simplifications} and \ref{simply}
apply again, since
$K_t(x,u) \lesssim e^{R(x)} < \alpha$ if $R(x) < (\log \alpha)/2$.
For $x\in\mathcal C_\alpha$, a combination of
 \eqref{tstort} and  \cite[Lemma 4.3{\rm{(i)}}]{CCS2}
implies
\begin{align*}
K_t
(x,u)
&\lesssim
e^{R(x)} \, \exp
 \big(
- c\,
\big|
\tilde u-\tilde x\big|^2
\big),
\end{align*}
where we use polar coordinates with $\beta = (\log\alpha)/2$.
The proposition now
follows from Lemma \ref{bounding-GMAa}.
\end{proof}

\begin{remark}
The inequality \eqref{3},
which is sharp as verified in \cite[Proposition 6.2]{CCS2}, is slightly stronger than the weak type $(1,1)$ estimate in
Theorem \ref{th-full-maximal}. The corresponding
estimate for the operator
$S_{0}^{{\rm{glob}}} $
 is false,
  since $f$ approximating a point mass at $0$ gives a counterexample.
\end{remark}

\begin{remark}
 In the case $Q=I$ and $B=-I$ an estimate similar to Lemma~\ref{sup small t}
  with a kernel $\overline M$ controlling from above the Mehler kernel $K_t$ in the global region, has recently been proved in \cite{Bruno}
 (see, in particular,   Definition~3.2 and Proposition~3.4
 therein).
 An earlier result of this type may be found in \cite[Proposition 2.1]{Soria-jlms}.
 These estimates are sharp for significant values of $(x,u)$, whereas our Theorem  \ref{th-full-maximal}
 is simpler, and sufficient for our needs.
  Moreover, Proposition    \ref{stima per M con t piccolovv}
 is stronger than the analogous bounds in \cite{Bruno}
 and \cite{Soria-jlms}.

\end{remark}


\begin{thebibliography}{MMMM}

\bibitem[ABQR]{Almeida}
V. Almeida, J. J. Betancor, P. Quijano and L. Rodr\'iguez-Mesa,
Maximal operator, Littlewood--Paley functions and variation operators associated with nonsymmetric Ornstein-Uhlenbeck operators,
{\em Mediterr. J. Math.} \textbf{20}, 196 (2023).


\bibitem[ABFQR]{Almeida2}
V. Almeida, J.J. Betancor, J.C. Fariña, P. Quijano, L. Rodríguez-Mesa, Littlewood--Paley functions associated with general Ornstein-Uhlenbeck semigroups,
	{\tt arXiv:2202.06136}. 


\bibitem[B]{Bernstein}
D. S. Bernstein,
{\em Matrix Mathematics. Theory, Facts, and Formulas}, (2009),
 Second Edition,
 Princeton University Press.


\bibitem[Bo]{Bogachev}
 V. I. Bogachev,
 Ornstein--Uhlenbeck operators and semigroups,
{\em Russ. Math. Surv.} (2018), \textbf{73}, 191--260.


\bibitem[Br]{Bruno}
T. Bruno,
Endpoint results for the Riesz transform of the Ornstein--Uhlenbeck operator,
{\em J. Fourier Anal. Appl.}, (2019), \textbf{25}, 1609--1631.

\bibitem[CaD]{Carbonaro-Oliver}
A. Carbonaro and O. Dragi\v cevi\'c,
Bounded holomorphic functional calculus for nonsymmetric Ornstein-Uhlenbeck operators,
 {\em Ann. Sc. Norm. Super. Pisa Cl. Sci.},
\textbf{XIX} (5) (2019), 1497--1533.



\bibitem[CCS1]{CCS1}
V. Casarino, P. Ciatti and P. Sj\"ogren,
The maximal operator of a normal  Ornstein-Uhlenbeck  semigroup is of  weak type (1,1),
 {\em Ann. Sc. Norm. Sup. Pisa Cl. Sci.} (5)  {\textbf{XXI}} (2020), 385-410.

\bibitem[CCS2]{CCS2}
 V. Casarino, P. Ciatti and P. Sj\"ogren,
{ On the maximal operator of a
general Ornstein--Uhlenbeck  semigroup}, {\em Math. Z.} \textbf{301} (2022), 2393--2413.

\bibitem[CCS3]{CCS3}
V. Casarino, P. Ciatti and P. Sj\"ogren,
{Riesz transforms of a general Ornstein--Uhlenbeck semigroup},
{\em Calculus of Variations and Partial Differential Equations},
\textbf{60}, 135 (2021).

\bibitem[CCS4]{CCS4}
V. Casarino, P. Ciatti and P. Sj\"ogren,
  On the orthogonality of generalized eigenspaces for the  Ornstein--Uhlenbeck  operator,
  {\em Arch. Math.}  \textbf{117}
(2021),
 547--556.

\bibitem[CCS5]{CCS6}
V. Casarino, P. Ciatti and P. Sj\"ogren, {On the variation operator for the Ornstein--Uhlenbeck  semigroup
in dimension one},  {\em Annali di Matematica Pura e Applicata},  (2023), https://doi.org/10.1007/s10231-023-01358-3.

  \bibitem[CDMY]{CDMY}
 M. Cowling,  I. Doust, A. McIntosh and  A. Yagi,
  Banach space operators with a bounded $H^\infty$ functional calculus,
  {\em J. Austral. Math. Soc. (Series A) }, \textbf{60}  (1996), 51--89.


\bibitem[CFMP1]{CFMP1}
R. Chill, E. Fa\v{s}angov\'a, G. Metafune and D. Pallara,
The sector of analyticity of the Ornstein--Uhlenbeck semigroup in $L^p$ spaces
with respect to invariant measure,
{\em J. Lond. Math. Soc.} \textbf{71}, (2005),  703--722.

\bibitem[CFMP2]{CFMP2}
R. Chill, E. Fa\v{s}angov\'a, G. Metafune and D. Pallara,
 The sector of analyticity of nonsymmetric submarkovian semigroups generated by elliptic operators,
 {\em C. R. Math. Acad. Sci. Paris} \textbf{342}, (2006), 909--914.

\bibitem[CiMP1]{CMP1}
    A. Cianchi, V. Musil and L. Pick,
     Moser inequalities in Gauss space,
{\em Math. Ann.}  (2020), \textbf{377}, 1265--1312.

\bibitem[CiMP2]{CMP2}
    A. Cianchi, V. Musil and L. Pick,
 On the Existence of Extremals for Moser-Type Inequalities in Gauss Space,
{\em Int. Math. Res. Not.} (2022),  \text{2022}, Issue 2,  1494--1537.

\bibitem[CiMP3]{CMP3}
    A. Cianchi, V. Musil and L. Pick,
 Sharp exponential inequalities for the Ornstein-Uhlenbeck operator,
{\em  J.  Funct. Anal.}
\textbf{281} (2021),  no. 11, Paper No. 109217.



\bibitem[EN]{Engel}
K. J. Engel and R. Nagel, {\em One parameter semigroups for linear evolution equations},
194
Springer Graduate Texts in Mathematics, 2000.

\bibitem[GMST1]{GCMSTRiesz}
J. Garc\'ia-Cuerva, G. Mauceri, P. Sj\"ogren and J. L. Torrea,
 Higher-order Riesz operators for the Ornstein-Uhlenbeck semigroup,
 {\em Potential Anal.},
{\textbf{10}}
 (1999),  379--407.


\bibitem[GMST2]{GCMST}
J. Garc\'ia-Cuerva, G. Mauceri, P. Sj\"ogren and J. L. Torrea,
Spectral multipliers for the Ornstein-Uhlenbeck semigroup,
 {\em  J. Anal. Math.},
{\textbf{78}}
 (1999),  281--305.


\bibitem[HP]{HP}

M. Haase and F. Pannasch,
Holomorphic H\"ormander-Type Functional Calculus on Sectors and Strips,
{\tt  	arXiv:2204.03936}.


\bibitem[K1]{K1}
M. Kemppainen,
An $L^1$-estimate for certain
spectral
multipliers
associated
with the Ornstein--Uhlenbeck
operator, {\em J. Fourier Anal. Appl.} (2016)
 \textbf{22}(6), 1416--1430.

\bibitem[K2]{K2}
M. Kemppainen,
  Admissible decomposition for spectral multipliers on Gaussian $L^p$,
{\em  Math. Z.} (2018) {\textbf{289}}, 983--994.


\bibitem[LB]{Lorenzi}
 L. Lorenzi and M. Bertoldi,
{\em Analytical methods for Markov semigroups},
Pure and Applied Mathematics (Boca Raton), 283, Chapman \& Hall/CRC, Boca Raton, FL, 2007.

\bibitem[M]{M}
A. McIntosh,
Operators which have an $H_\infty$ functional calculus, {\em Miniconference
on operator theory and partial differential equations}, (Proc.\ Centre Math. Analysis 24, ANU, Canberra, 1986), 210--231.

\bibitem[MPS]{Soria-jlms}
 T. Men\'arguez,
S. P\'erez and F. Soria,
The Mehler maximal function: a geometric proof of the weak type 1,
{\em J. London Math. Soc.}
(2000), \textbf{61},
 846--856.

\bibitem[MPP]{MPP}
G.Metafune, D. Pallara and E. Priola,
Spectrum of Ornstein-Uhlenbeck operators in $L^p$ spaces with respect to invariant measures,
{\em{J.  Funct. Anal.}} \textbf{196} (2002), 40--60.


\bibitem[MPRS]{MPRS}
G. Metafune, J. Pr\"uss, A. Rhandi, and R. Schnaubelt,
The domain of the Ornstein-Uhlenbeck
operator on a $L^p$-space with invariant measure,
{\em Ann. Sc. Norm. Super. Pisa Cl. Sci.} \textbf{1}, (2002)   471--487.

\bibitem[OU]{OU}
L. S. Ornstein and G. E.
Uhlenbeck,
 On the theory of Brownian Motion,
{\em Phys. Rev.} \textbf{36},  (1930) 823--841.

\bibitem[P]{P}
F. Pannasch,
The Holomorphic H\"ormander Functional Calculus. Ph.D. Thesis, Kiel University, 2019,
{\tt https:\/\/d-nb.info/1202630561/34}.



\bibitem[St]{Stein-LP}
E. M. Stein,
{Topics in harmonic analysis related to the Littlewood-Paley theory}, Annals of Mathematics Studies,
No. 63 Princeton University Press, Princeton, N.J..    




\bibitem[U]{Urbina}
W. Urbina-Romero,
{\emph {Gaussian Harmonic Analysis}}, Springer Monographs in Mathematics, 2019.



\bibitem[W1]{Wr1}
B. Wr\'obel, Multivariate spectral multipliers for systems of Ornstein-Uhlenbeck operators,
{\em Studia
Math.} (2013), \textbf{216}, 47--67.

\bibitem[W2]{Wr2}
B. Wr\'obel,
Joint Spectral Multipliers for Mixed Systems of Operators,
{\em J.Fourier Anal. Appl.} (2017), \textbf{23}, 245--287.


\bibitem[W3]{Wr3}
B. Wr\'obel,
Laplace Type Multipliers for Laguerre Expansions of Hermite Type,
{\em Mediterr. J. Math.}  {\textbf{10}}, (2013), 1867--1881.


\end{thebibliography}
\end{document}